\documentclass[11pt]{amsart}
\usepackage{amssymb}
\usepackage{pgfplots}
\usepackage{csvsimple}
\usepackage{siunitx}
\usepackage{hyperref}
\usepackage{graphicx,color}
\usepackage[normalem]{ulem}
\usepackage{mathrsfs}
\def\R{\mathbb{R}}

\def\cA{\mathcal{A}}

\def\cI{\mathcal{I}}

\def\cM{\mathcal{M}}

\def\cS{\mathcal{S}}
\def\cT{\mathcal{T}}

\def\a{\alpha}
\def\b{\beta}
\def\g{\gamma}
\def\d{\delta}

\def\k{\kappa}
\def\l{\lambda}
\def\s{\sigma}
\def\p{\partial}

\def\veps{\varepsilon}

\def\G{\Gamma}

\def\wto{\rightharpoonup}

\def\transp{{\sf T}}

\def\hy{\widehat{y}}

\newcommand{\dv}[1]{\,{\mathrm d}#1}

\newcommand{\wcheck}[1]{#1\hspace{-.8ex}\mbox{\huge {\lower.45ex \hbox{$\textstyle \check{}$}}} \hspace{.5ex}}

\let\oldmarginpar\marginpar
\renewcommand\marginpar[1]{
  \oldmarginpar[\raggedleft\footnotesize #1]
  {\raggedright\footnotesize #1}}
\newtheorem{definition}{Definition}
\newtheorem{lemma}[definition]{Lemma}
\newtheorem{proposition}[definition]{Proposition}

\newtheorem{corollary}[definition]{Corollary}
\newtheorem{remark}[definition]{Remark}

\newtheorem{example}[definition]{Example}

\newtheorem{algorithm}[definition]{Algorithm}

\numberwithin{definition}{section}

\def\oQ{\overline{Q}}

\def\ta{\widetilde{a}}
\def\tb{\widetilde{b}}
\def\talpha{\widetilde{\a}}
\def\tbeta{\widetilde{\b}}
\def\ha{\widehat{a}}
\def\hb{\widehat{b}}
\def\ts{\widetilde{s}}

\def\bc{{\rm bc}}
\def\frame{{\rm frame}}
\def\inv{{\rm inv}}

\begin{document}
\title{Numerical simulation of inextensible elastic ribbons}
\author{S\"oren Bartels}
\address{Abteilung f\"ur Angewandte Mathematik,  
Albert-Ludwigs-Universit\"at Freiburg, Hermann-Herder-Str.~10, 
79104 Freiburg i.~Br., Germany}
\email{bartels@mathematik.uni-freiburg.de}
\date{\today}
\renewcommand{\subjclassname}{%
\textup{2010} Mathematics Subject Classification}
\begin{abstract}
Using dimensionally reduced models for the numerical simulation 
of thin objects is highly attractive as this reduces the computational
work substantially. The case of narrow thin elastic bands is considered 
and a convergent finite element discretization for the one-dimensional
energy functional together with a fully practical, energy-monotone 
iterative method for computing stationary configurations are devised. 
Numerical experiments confirm the theoretical findings and illustrate
the qualitative behavior of elastic narrow bands.
\end{abstract}
 

\maketitle
 
\section{Introduction}
Thin elastic objects provide a fascinating range of applications, particularly
in nanotechnology, where small lightweight constructions can undergo large
deformations that are controlled by external stimuli, cf., e.g.,~\cite{SZTDI12}. 
An important class of
these objects are elastic ribbons which are thin elastic bands of small width.
A characteristic feature is that these bands are unshearable and hence curvature
can only occur orthogonal to the strip's inplane direction.  A particular 
configuration that has been addressed in the literature are M\"obius strips of
small width, cf.~\cite{Wund62,StaHei07,BarHor15}.

To effectively simulate the highly nonlinear behavior of thin elastic objects
dimensionally reduced mathematical descriptions are necessary. In the case of
ribbons a one-dimensional model has been proposed by Sadowsky in~\cite{Sado30a,Sado30b}
which was obtained via formal arguments; further justifications have been given
in~\cite{Wund62,KirFri15}. A complete, mathematically rigorous derivation 
from three-dimensional hyperelasticity has been carried out in~\cite{FHMP16}, which 
shows that a correction of Sadowsky's functional is necessary. The correction
concerns the important case of vanishing curvature and thereby provides a more
general description. Additionally, this eliminates a singularity in 
Sadowsky's functional which is particularly important for the stable
discretization and iterative solution. The nonsmooth character of the
corrected functional underlines the expected feature that the behavior of ribbons is 
substantially more complicated than that of elastic rods, cf.~\cite{MorMul03} 
for a corresponding rigorously derived mathematical model. Our numerical experiments
confirm that the correction in the one-dimensional functional is relevant.
A comparison of certain instabilities of twisted rods and ribbons is provided in~\cite{AudSef15}.
The presence of such effects supports the importance of developing stable numerical 
approximation schemes. 

The model identified in~\cite{FHMP16} describes the deformation of an elastic ribbon
of length $L>0$ via a frame $(y,b,d)$ consisting of functions 
\[
y,b,d : (0,L) \to \R^3,
\]
defined on the centerline $(0,L)$ of the undeformed ribbons. The frame condition
requires the functions to satisfy
\[
[y',b,d] \in SO(3)
\]
pointwise in $(0,L)$ with the set of proper orthogonal matrices~$SO(3)$. 
In particular, $y$ provides an arclength parametrization of 
the deformed centerline. The arclength property $|y'|^2=1$ models the physical
feature that thin elastic objects are inextensible in the bending regime. 
The directors~$b$ and~$d$ define the deformation of the band orthogonal to the 
centerline. Since the band cannot be sheared, the in-plane curvature 
vanishes, i.e., the additional condition
\[
y'' \cdot b = 0
\]
occurs, which in view of the relation $y'\cdot b=0$ is equivalent to the
condition $y'\cdot b'=0$. A frame $(y,b,d)$ describing the deformation of an
elastic band for given boundary conditions minimizes the dimensionally reduced
elastic energy
\[
E[y,b,d] = \int_0^L \oQ(y''\cdot d,b'\cdot d) \dv{x}.
\]
Appropriately scaled body forces can be included in the minimization problem. 
Because of the frame condition and the constraint $y''\cdot b = 0$,
the functional can be expressed in terms of the curvature $\k$ and torsion $\tau$
of the curve parametrized by~$y$ and described by the frame $(y,b,d)$. These
quantities are given by 
\[
\k^2 = (y''\cdot d)^2 = |y''|^2, \quad \tau^2 = (b'\cdot d)^2 = |b'|^2
\]
and we may consider the equivalent functional
\[
E[y,b] = \int_0^L \oQ(|y''|,|b'|) \dv{x}.
\]
The omitted director~$d$ can be reconstructed via $d=y'\times b$. 

Various aspects of the outlined variational problem make the numerical 
approximation of minimizers difficult. First, the function $\oQ$ is convex and
continuously differentiable but not twice continuously differentiable.
Hence, to apply iterative solution techniques,
regularizations are required. Second, the nonlinear pointwise constraints
have to be discretized appropriately and incorporated in the solution 
strategy. Third, the nonlinear dependence of the integrand on highest
derivatives complicates the analysis of quadrature effects. Fourth, the
iterative minimization has to be carefully done to reliably decrease the
energy and avoid the occurrence of irregular stationary configurations. 

We devise a discretization of $E$ that uses piecewise cubic, $C^1$ conforming
finite elements for the approximation of $y$, and piecewise linear, 
$C^0$ conforming finite elements for $b$. The variable $d$ is eliminated
via the relation $d=y'\times b$. A twice continuously differentiable 
regularization $\oQ_\d$ of
$\oQ$ is used and the orthogonality relations $y'\cdot b = 0$ and
$y'\cdot b'=0$ are incorporated using penalization, i.e., we consider
\[\begin{split}
E_{\d,\veps,h}[y_h,b_h] = \int_0^L & \oQ_\d(|y_h''|,|b_h'|) \dv{x} \\
&+ \frac{1}{2\veps_1} \int_0^L (y_h'\cdot b_h)^2 \dv{x}
+ \frac{1}{2\veps_2} \int_0^L (y_h'\cdot b_h')^2 \dv{x}.
\end{split}\]
The unit length conditions $|y'|^2=1$ and $|b|^2=1$ are imposed at
the nodes of the partitioning defining the finite element spaces, i.e.,
\[
\cI_h^{1,0} |y_h'|^2 = \cI_h^{1,0} |b_h|^2 = 1
\]
with the nodal interpolation operator $\cI_h^{1,0}$ related to piecewise
linear, continuous finite element functions. We derive three convergence 
results for the discretized functional and its minimizers. Under an
approximability condition on minimizers of the continuous problem we
show that these are correctly approximated by discrete minimizers. 
{Assuming the density of smooth ribbon frames or a regularity property of a 
minimizer} we are able to verify
the convergence of the discrete fucntionals in the sense of $\G$-convergence
to the corrected Sadowsky functional. 
Most generally, we deduce that the functionals $E_{\d,\veps,h}$ are
$\G$-convergent to the corrected Sadowsky functional with a penalized
treatment of the orthogonality relation $y'\cdot b' =0$, i.e., for a fixed
parameter $\veps_2$.  A generalization of the results based on Jensen's 
inequality includes the use
of quadrature in nonquadratic terms which makes the devised method
fully practical even in the case of minimial regularity properties. 
{Justifications of the assumed density and regularity results are currently
under investigation~\cite{FHMP-in-prep-pre}.}

The discretized elastic energies $E_{\d,\veps,h}$ are particularly suitable
for the minimization using gradient flows, i.e., a discrete version of the
formal evolution equations
\[
\p_t y = -\nabla_y E[y,b] + (\l y')', \quad \p_t b = - \nabla_b E[y,b] - \mu b,
\]
for given initial and boundary data and suitable definitions of the gradients with
respect to the variables $y$ and $b$. The Lagrange multipliers $\l$ and $\mu$
correspond to the pointwise unit length constraints. To define a discrete
variant we use a splitting of the (regularized) energy density $\oQ_\d$ into
quadratic and nonlinear parts, i.e., we have
\[
\oQ_\d(\a,\b) = \frac12 \a^2 + \frac{5}{2} \b^2 + \psi_\d(\a^2,\b^2).
\]
With
backward difference quotient operator~$d_t$ to approximate the time derivatives 
and suitable implicit and explicit treatments of variations of the
discrete energy $E_{\d,\veps,h}$, i.e., the time stepping scheme reads
\[\begin{split}
 (d_t y_k,w)_\star + &(y_k'',w'') +  \veps_1^{-1} (y_k'\cdot b_{k-1},w'\cdot b_{k-1})
+ \veps_2^{-1} (y_k'\cdot b_{k-1}',w'\cdot b_{k-1}') \\ 
& = - (\psi_{\d,1}'(|y_{k-1}''|^2,|b_{k-1}'|^2) y_{k-1}'',w''), \\
 (d_t b_k,r)_\dagger + & 5 (b_k',r') +  \veps_1^{-1} (y_k'\cdot b_k,y_k'\cdot r)
+ \veps_2^{-1} (y_k'\cdot b_k',y_k'\cdot r') \\
& = - (\psi_{\d,2}'(|y_{k-1}''|^2,|b_{k-1}'|^2) b_{k-1}',r'),
\end{split}\]
subject to linearized pointwise constraints 
\[
d_t y_k' \cdot y_{k-1}' = 0, \quad d_t b_k \cdot b_{k-1} = 0.
\]
The equations are linear and decoupled in the variables~$y_k$ and~$b_k$.
For this decoupling the penalized approximation of the orthogonality relations
using separately convex functionals is essential. We show that the
resulting discrete time-stepping scheme is energy decreasing and 
that the violation of the constraints is controlled uniformly.

An alternative approach to the numerical simulation of elastic ribbons
is obtained by discretizing the two-dimensional nonlinear plate bending model 
from~\cite{FrJaMu02} and the finite element method devised in~\cite{Bart13,Bart15-book}. 
This has the disadvantage of discretizing a
two-dimensional domain but the advantages that the model leads to
the constrained minimization of a quadratic functional and incorporates
the width parameter which may be relevant in certain applications. However,
the explicit occurrence of the width requires the use of very 
fine discretizations.

The outline of this article is as follows. We specify employed notation and
elementary facts about finite element methods in Section~\ref{sec:prelim}.
The continuous variational problem and its regularized approximation
are addressed in Section~\ref{sec:reduced}. In Section~\ref{sec:fem} we
define a discrete approximation of the regularized variational problem
and investigate its convergence to the continuous one. A discrete 
gradient flow is devised and analyzed in Section~\ref{sec:iterative}.
Numerical experiments are reported in Section~\ref{sec:num_ex}.

\section{Preliminaries}\label{sec:prelim}
\subsection{Finite element spaces}
We always let $\cT_h$ denote a partitioning $x_0<x_1<\dots<x_N$ 
of the interval $(0,L)$ into elements $T_j = [x_{j-1},x_j]$, $j=1,2,\dots,N$,
with  maximal mesh-size $h=\max_{j=1,\dots,N} h_{T_j}$ with 
$h_{T_j}= |x_j-x_{j-1}|$. Sets of piecewise polynomial functions of 
degree at most $p$ 
and which are $k$~times continuously differentiable are denoted
by the sets
\[
\cS^{p,k}(\cT_h) = 
\big\{v_h \in C^k(0,L): v_h|_T \in P_p(T) \text{ for all } T\in \cT_h\big\}.
\]
Associated with these spaces is a nodal interpolation operator
\[
\cI_h^{p,k}: C^k(0,L)\to \cS^{p,k}(\cT_h).
\]
We work with the spaces $\cS^{1,0}(\cT_h)$ of piecewise linear, globally 
continuous functions and $\cS^{3,1}(\cT_h)$ of piecewise cubic, globally 
continuously differentiable functions for which the interpolation 
operators are entirely determined by the conditions 
\[
\cI_h^{1,0}v(x_j) = v(x_j)
\]
and 
\[
\cI_h^{3,1}w (x_j) = w(x_j), \quad \big(\cI_h^{3,1}w \big)'(x_j) = w'(x_j)
\]
for $v\in C^0([0,L])$ and $w\in C^1([0,L])$ and
$j=0,1,\dots,N$, respectively. The operators satisfy the
estimates
\[
\| v- \cI_h^{1,0} v\|_{H^r(0,L)}  \le c_{1,0} h^{2-r} \|v''\|,
\]
for $r=0,1$, and
\[
\| w- \cI_h^{3,1} w\|_{H^r(0,L)}  \le c_{3,1}h^{3-r} \|v^{(3)}\|,
\]
for $r=1,2$, respectively. Moreover, 
$\|(\cI_h^{3,1} w)'''\|_{L^2(T)}  \le c_{3,1}' \|w^{(3)}\|_{L^2(T)}$ 
for all $T\in \cT_h$ and $w\in H^3(T)$. 
We also note the inverse estimate
\begin{equation}\label{eq:inv_est}
\|v_h\|_{L^\infty(T)} \le c_\inv h_T^{-1/2} \|v_h\|_{L^2(T)}
\end{equation}
for all $T\in \cT_h$ and $v_h\in P_p(T)$ with a constant $c_\inv$ 
depending on the polynomial degree~$p$. We will always assume that
the partitioning $\cT_h$ is quasiuniform so that the minimal mesh-size
is comparable to the maximal mesh-size, i.e., $\min_{T\in\cT_h} h_T
\ge c h$. 

\subsection{Difference quotients}
Given a step-size $\tau>0$ and an arbitrary sequence $(a^k)_{k=0,\dots,K}$, 
the backward 
difference quotient $d_t a^k$ is for $k=1,2,\dots,K$ defined via
\[
d_t a^k = \tau^{-1} \big(a^k-a^{k-1}\big).
\]
A binomial formula implies the identity
\[
d_t a^k \cdot a^k = \frac12 d_t |a^k|^2 + \frac{\tau}{2} |d_t a^k|^2.
\]

\subsection{Lebesgue and Sobolev spaces}
Throughout this article we use standard notation for Sobolev and Lebesgue
spaces. We abbreviate the $L^2$ norm via
\[
\|\cdot \| = \|\cdot \|_{L^2(0,L)}.
\]
Throughout
we use standard embedding results, e.g., that $H^1(0,L)$ is compactly
embedded in $C^0(0,L)$. 

\section{Reduced model and regularized approximation}\label{sec:reduced}
\subsection{Elastic ribbons}
We consider a dimensionally reduced model for narrow ribbons rigorously derived
in~\cite{FHMP16}. The one-dimensional variational problem is defined via the elastic 
energy 
\[
E[y,b] = \int_0^L \oQ(|y''|,|b'|) \dv{x}.
\]
The vectors $y'$, $b$, and $d=y'\times b$ are required to constitute an orthonormal
frame, i.e., 
\[
[y',b,d]\in SO(3)
\]
almost everywhere in $(0,L)$. This implies that $y$ and $b$ satisfy the constraints
\[
|y'|^2 = 1, \quad |b|^2 =1, \quad y'\cdot b =0.
\]
Moreover, the derivation of the model leads to the constraint that the inplane 
curvature component vanishes, i.e., 
\[
y'' \cdot b = 0,
\]
which in view of $y'\cdot b=0$ is equivalent to $y'\cdot b'=0$.
The function $\oQ$ is given by
\[
\oQ(\a,\b) = 
\begin{cases}
(\a^2 + \b^2)^2/\a^2 & \mbox{if } |\a| \ge |\b|, \\
4 \b^2 & \mbox{if } |\a| \le |\b|.
\end{cases}
\]
The following lemma provides structural properties of $\oQ$. 

\begin{lemma}[Convexity]\label{la:convex_and_diffb}
The function~$\oQ:\R_{\ge 0}^2 \to \R$ is continuously differentiable and convex so
that the mapping 
\[
(z,r) \mapsto \int_0^L \oQ(|z|,|r|) \dv{x}
\]
is weakly lower semicontinuous on $L^2(0,L;\R^3)\times L^2(0,L;\R^3)$. 
\end{lemma}

\begin{proof}
For $\a\ge \b$ we have
\[
\nabla \oQ(\a,\b) = \begin{bmatrix} 2 \a - 2 \b^4/\a^3 \\ 4 \b + 4 \b^3/\a^2 \end{bmatrix},
\]
and for $\a\le \b$ that 
\[
\nabla \oQ(\a,\b) = \begin{bmatrix} 0 \\ 8 \b \end{bmatrix},
\]
so that~$\oQ$ is continuously differentiable. 
It is obvious that $D^2 \oQ$ is positive semi-definite if $\a\le \b$. 
Letting $q=\b/\a$ the Hessian $D^2\oQ$ is for $\a\ge \b$ given by
\[
D^2 \oQ(\a,\b) = \begin{bmatrix} 2 + 6 q^4 & - 8 q^3 \\ - 8 q^3 & 4 + 12 q^2 \end{bmatrix}.
\]
For $x=(x_1,x_2)\in \R^2$ we have, using $q^2 \le 1$,
\[\begin{split}
x^\transp D^2 \oQ(\a,\b) x 
&= 2 x_1^2 + 4 x_2^2 + q^2 ( 6 q^2 x_1^2 - 16 q x_1 x_2 + 12 x_2^2)  \\
&\ge q^2 (8 q^2 x_1^2 - 16 q x_1 x_2 + 16 x_2^2) \ge 0,
\end{split}\]
since, e.g., $16 q x_1 x_2 \le 4 q^2 x_1^2 + 16 x_2^2$. A Taylor expansion yields the convexity
property $\oQ(\talpha,\tbeta) \ge \oQ(\a,\b) + \nabla \oQ(\a,\b) \cdot ((\talpha,\tbeta)-(\a,\b))$.  
\end{proof}

\subsection{Regularization}
A smooth approximation of the function~$\oQ$ is obtained by using
a regularized modulus function, i.e., we assume that for every $\d\ge 0$ we
are given a function 
\[
|\cdot|_\d : \R \to \R_{\ge 0}
\]
that satisfies $|\cdot|_\d \in C^2(\R)$, obeys the estimates
\begin{equation}\label{eq:mod_reg}
|\cdot|_\d \ge c_0 \d, \quad |\cdot|_\d' \le c_1, \quad |\cdot|_\d'' \le c_2 \d^{-1},
\end{equation}
and approximates the modulus function 
uniformly, i.e., for all $x\in \R$ we have 
\begin{equation}\label{eq:uni_reg}
\big||x|_\d -|x|\big| \le c_{{\rm uni}} \d.
\end{equation}
A possible choice is the function $|x|_\d = (x^2+\d^2)^{1/2}$.
Motivated by the relation $2 \max\{x,y\} = x + y + |x-y|$ we define
\[
\oQ_\d(\a,\b) = 2\b^2 +  \frac12 \a^2 + \frac12 \b^2 + \frac12 |\a^2-\b^2|_\d + 
\frac{2 \b^4}{ \a^2 + \b^2 + |\a^2-\b^2|_\d},
\]
and note that $\oQ_\d$ coincides with $\oQ$ for $\d=0$. 
To prove the consistency of our discretization, that will be based on the
regularization, we note the following estimate.

\begin{lemma}[Local Lipschitz continuity]\label{la:cont_reg}
Given $\a,\b,\talpha,\tbeta \in \R$ we have 
\[\begin{split}
\big|\oQ(\a,\b) - \oQ_\d (\talpha,\tbeta)\big| 
 \le 21 \big(|\b-\tbeta||\b+\tbeta| + |\a-\talpha||\a+\talpha|+ c_{{\rm uni}} \d\big).
\end{split}\]
\end{lemma}

\begin{proof}
We define 
\[
s^2 = \a^2+\b^2+|\a^2-\b^2|, \quad
\ts^2 = \talpha^2+\tbeta^2+|\talpha^2-\tbeta^2|_\d,
\]
and note that 
\[
\oQ(\a,\b) = 2 \b^2 + \frac12 s^2 + \frac{2 \b^4}{s^2}, \quad
\oQ_\d(\talpha,\tbeta) = 2 \tbeta^2 + \frac12 \ts^2 + \frac{2 \tbeta^4}{\ts^2}.
\]
Straightforward calculations lead to 
\[
\frac{\b^4}{s^2} -\frac{\tbeta^4}{\ts^2} 
= \frac{(\ts^2-s^2)\b^4}{s^2\ts^2} + \frac{(\b^2-\tbeta^2)(\b^2+\tbeta^2)}{\ts^2},
\]
and, by exchanging the roles of the variables, in case $s^2>0$, 
\[
\frac{\tbeta^4}{\ts^2} - \frac{\b^4}{s^2}
= \frac{(s^2-\ts^2) \tbeta^4}{\ts^2 s^2} + \frac{(\tbeta^2-\b^2)(\tbeta^2+\b^2)}{s^2}.
\]
On combining the two identities, we find that
\[
\Big| \frac{\b^4}{s^2} -\frac{\tbeta^4}{\ts^2} \Big|
\le 
\big(|\ts^2-s^2|+ |\b^2-\tbeta^2| \big) 
\min \Big\{ \frac{\b^4}{s^2\ts^2} + \frac{\b^2+\tbeta^2}{\ts^2},
\frac{\tbeta^4}{s^2\ts^2} + \frac{\b^2+\tbeta^2}{s^2}\Big\}.
\]
Noting that $s^2 \ge \b^2$
and $\ts^2\ge \tbeta^2 + c_0 \d \ge \tbeta^2$, the estimate
implies that we have
\[
\Big| \frac{\b^4}{s^2} -\frac{\tbeta^4}{\ts^2} \Big|
\le \big(|\ts^2-s^2|+ |\b^2-\tbeta^2| \big) 
\min \Big\{1+2 \frac{\b^2}{\tbeta^2+c_0\d}, 1+ 2 \frac{\tbeta^2}{\b^2} \Big\}.
\]
A case distinction relating $\b^2$ and $\tbeta^2+c_0\d$
shows that the minimum on the right-hand side is always bounded by~3.
Using~\eqref{eq:uni_reg}, Lipschitz continuity of the modulus function, 
and binomial formulas, show that we have
\[
\big|s^2 -\ts^2\big| \le 2 |\a-\talpha||\a+\talpha| + 2 |\b-\tbeta||\b+\tbeta| + c_{{\rm uni}} \d.
\]
Combining the estimates implies that we have
\[\begin{split}
\big|\oQ(\a,\b) - & \oQ_\d (\talpha,\tbeta)\big| 
\le 2 |\b^2-\tbeta^2| + \frac12 |s^2-\ts^2| + 2 \Big|\frac{\b^4}{s^2}- \frac{\tbeta^4}{\ts^2}\Big| \\
&\le (2+6) |\b^2-\tbeta^2| + \big(\frac12 + 6\big) |s^2-\ts^2|  \\
& \le |\b-\tbeta||\b+\tbeta|\big(2+6+13 \big) + 13|\a-\talpha||\a+\talpha| 
+ \frac{13}{2} c_{{\rm uni}} \d.
\end{split}\]
This implies the asserted estimate. 
\end{proof}

\subsection{Approximation}
Lemma~\ref{la:cont_reg} implies the uniform convergence property
$\oQ_\d \to \oQ$. Another consequence is the strong approximability
of the variational problem defined by~$E$. More generally, we consider the
approximation by variational problems with energy functionals~$E_{\d,\veps_2}$ 
using the regularized function $\oQ_\d$
and a penalized treatment of the relation $y'\cdot b'=-y''\cdot b =0$, 
i.e., 
\[
E_{\d,\veps_2}[y,b] = \int_0^L \oQ_\d(|y''|,|b'|) \dv{x} 
+ \frac{1}{2\veps_2} \int_0^L (y'\cdot b')^2 \dv{x}.
\]
We incorporate boundary conditions via a bounded linear operator $L_\bc$,
which is assumed to depend only on the boundary values of $y$, $y'$, and $b$. 
It defines the set of admissible frames
\[\begin{split}
\cA_\frame = \big\{  (y,b) &\in H^2(0,L;\R^3)\times H^1(0,L;\R^3) : \\
& [y',b,y'\times b]\in SO(3) \ \text{a.e. in } (0,L), \  L_\bc[y,b] = \ell_\bc \big\}.
\end{split}\]
Note that the relation $y'\cdot b'=0$ is not included in the set~$\cA_\frame$. 
The variational problem under consideration then reads as follows.
\begin{equation}\label{eq:cont_prob}
\tag{$\textrm{P}_{\d,\veps_2}$}
\left\{
\begin{array}{l}
\text{Given $\veps_2,\d\ge 0$ find a minimizing pair} \\
\text{$(y,b) \in \cA_\frame$ for } E_{\d,\veps_2}[y,b].
\end{array}
\right.
\end{equation}
The original problem formally corresponds to $\d=0$ and $\veps_2=0$ when the 
penalty term is interpreted as a strict constraint. Due to the convexity of
$\oQ$ a solution exists for $\d=0$ and $\veps_2 \ge 0$. 
Our numerical scheme will be a discrete variant of the following approximation result.

\begin{proposition}[Strong approximability]\label{prop:strong_approx}
For every $(y,b) \in \cA_\frame$ with $y'\cdot b' = 0$ 
and every $\s>0$ there exists
\[
(y_\s,b_\s)\in \big[C^\infty(0,L;\R^3)\big]^2 \cap \cA_\frame
\]
with
\[
\|y-y_\s\|_{H^2(0,L)} + \|b-b_\s\|_{H^1(0,L)} 
+ \|y_\s'\cdot b_\s'\| \le \s,
\]
such that 
\[
\big| E[y,b] - E_{\d,\veps_2}[y_\s,b_\s]\big| \le
c_{{\rm sa}} \big(\s^2\veps_2^{-1} +\d+\s\big).
\]
In particular, the functionals $\big(E_{\d,\veps_2}\big)_{\d,\veps_2}$ are
$\G$-convergent to $E$ for $(\d,\veps_2)\to 0$ and to $E_{0,\veps_2}$ for
fixed $\veps_2>0$ and $\d \to 0$. 
\end{proposition}

\begin{proof}
A regularized frame with the asserted properties 
can be constructed using mollification and correction 
techniques,~cf.~\cite{BarRei19-in-prep-pre}.
Lemma~\ref{la:cont_reg} and the identity 
$y_\s'\cdot b_\s' = (y_\s-y)'\cdot b_\s'+ y'\cdot (b_\s-b)'$ 
then imply the approximation result. The $\G$-convergence results
are immediate consequences of the estimates. 
\end{proof}

\begin{remark}\label{rem:explode}
We remark that the quantity
\[
[y,b]_\s = \|y_\s\|_{H^3(0,L)} + \|b_\s\|_{H^2(0,L)} 
\]
is in general unbounded as $\s\to 0$. 
If $(y,b) \in H^3(0,L;\R^3)\times H^2(0,L;\R^3)$ then it remains bounded. 
\end{remark}

\section{Finite element approximation}\label{sec:fem}
We write the regularized minimization problem with a penalized treatment of
the orthogonality relations as 
\[\begin{split}
E_{\d,\veps}[y,b] &= \int_0^L \oQ_\d(|y''|,|b'|) \dv{x} + P_\veps[y,b] \\
&= \frac12 \int_0^L |y''|^2 + 5 |b'|^2 + \psi_\d(|y''|^2,|b'|^2) \dv{x} + P_\veps[y,b],
\end{split}\]
where for $\veps = (\veps_1,\veps_2)$ we define
\[
P_\veps[y,b] = \frac{1}{2\veps_1} \int_0^L (y'\cdot b)^2 \dv{x}
+ \frac{1}{2\veps_2} \int_0^L (y'\cdot b')^2 \dv{x},
\]
and the function $\psi_\d$ is defined via
\[
\psi_\d(\a^2,\b^2) = |\a^2-\b^2|_\d + \frac{4 \b^4}{\a^2+\b^2 + |\a^2-\b^2|_\d}.
\]
The minimization of $E_{\d,\veps}$ is subject to the pointwise unit length conditions 
\[
|y'|^2 = 1, \quad |b|^2 = 1.
\]
The corresponding discrete problem uses the conforming subspaces
\[\begin{split}
V_h & = \cS^{3,1}(\cT_h)^3 \times \cS^{1,0}(\cT_h)^3 \\
&\subset V = H^2(0,L;\R^3)\times H^1(0,L;\R^3).
\end{split}\]
By imposing the  nonlinear constraints at the nodes of the partitioning $\cT_h$ we
arrive at the following discrete problem. 
\begin{equation}\label{eq:discr_prob}
\tag{$\textrm{P}_{\d,\veps,h}$}
\left\{ \begin{array}{l}
\text{Find a minimizing pair $(y_h,b_h) \in V_h$ for} \\[1.5mm]
\displaystyle{ 
E_{\d,\veps,h}[y_h,b_h] 
= \frac12 \int_0^L |y_h''|^2 + 5 |b_h'|^2 + \psi_\d(|y''|^2,|b'|^2) \dv{x}} \\
\displaystyle{ \hspace*{3.5cm} + P_\veps[y_h,b_h]} \\[3mm]
\text{subject to} 
\displaystyle{\quad \cI_h^{1,0} |y_h'|^2 = \cI_h^{1,0} |b_h|^2 = 1 \text{ and } L_\bc[y_h,b_h] = \ell_\bc.}
\end{array}\right.
\end{equation}
The following proposition implies that discrete minimizers converge to
minimizers of the continuous problem provided that these satisfy an
approximability condition. Because of the assumption that the 
operator $L_\bc$ only depends on the boundary values of $y$, $y'$, and $b$,
these are preserved
under the nodal interopolation operators $\cI_h^{3,1}$ and $\cI_h^{1,0}$. 

\begin{proposition}[Partial $\G$-convergence]\label{prop:gamma_conv}
The discrete 
constrained energy functionals $E_{\veps,\d,h}$ approximate the 
continuous constrained energy functional $E$ in the sense of the following
assertions. \\
(i) Let $(y,b)\in V$. If $(\s,\veps_1,\veps_2,\d) \to 0$ as $h\to 0$
with $\s>0$ as in Proposition~\ref{prop:strong_approx} such that
\[
h^2 \veps_1^{-1} + \s^2 \veps_2^{-1} + h^2 [y,b]_\s^2 \veps_2^{-1} \to 0
\]
then there exists a sequence
of discrete admissible pairs $(y_h,b_h) \in V_h$ with $y_h \to y$
in $H^2$ and $b_h\to b$ in $H^1$ as $h\to 0$ and such that
\[
E_{\d,\veps,h}[y_h,b_h] \to E[y,b].
\]
(ii) For every bounded sequence $(y_h,b_h)\in V_h$ of admissible 
pairs with weak limit $(y,b) \in V$ we have that the limitting pair
is admissible and that 
\[
E[y,b] \le \liminf_{(\veps,\d,h)\to 0} E_{\d,\veps,h}[y_h,b_h].
\]
\end{proposition}

\begin{proof}
(i) We use the result of Proposition~\ref{prop:strong_approx} and consider
for $\s>0$ the smooth frame $(y_\s,b_\s)\in \cA_\frame$ that
approximates the general frame $(y,b) \in \cA_\frame$. We then
consider for $h>0$ the pair $(y_h,b_h)$ defined via
\[
y_h = \cI_h^{3,1} y_\s, \quad b_h = \cI_h^{1,0} b_\s.
\]
The pair $(y_h,b_h)$ is admissible in the discrete minimization problem
 i.e., the functions satisfy the boundary conditions and the nodal unit
length constraints. Since $y_\s'\cdot b_\s= 0$ we find that
\[
y_h' \cdot b_h = (y_h'-y_\s')\cdot b_h + y_\s' \cdot (b_h-b_\s),
\]
and hence, we have
\[
\|y_h' \cdot b_h\| \le c h.
\]
Similarly, with the quantity $c_\s= [y,b]_\s$ from Remark~\ref{rem:explode},
that bounds the norms
$\|y_\s\|_{H^3(0,L)}$ and $\|b_\s\|_{H^2(0,L)}$, we have that
\[
\|y_h'\cdot b_h'\| \le \s + c h c_\s.
\]
{The repeated application of Lemma~\ref{la:cont_reg} provides the estimate
\[\begin{split}
\big| & \oQ_\d(\a,\b) - \oQ_\d(\ta,\tb)\big| \\
&\le \big|\oQ_\d(\a,\b) - \oQ(\a,\b)\big| 
+ \big|\oQ(\a,\b) - \oQ(\ta,\tb)\big| 
+\big|\oQ_\d(\ta,\tb) - \oQ_\d(\ta,\tb)\big|  \\
&\le  c \big(|\b-\tbeta||\b+\tbeta| + |\a-\talpha||\a+\talpha| + \d\big).
\end{split}\]
In combination with Proposition~\ref{prop:strong_approx} this yields that}
\[\begin{split}
\big|E[y,b]& -E_{\d,\veps,h}[y_h,b_h] \big| \\
&\le \big|E[y,b] - E_{\d,\veps_2}[y_\s,b_\s] \big|
  + \big|E_{\d,\veps_2}[y_\s,b_\s]-E_{\d,\veps,h}[y_h,b_h]  \big| \\
&\le c \big(\s^2 \veps_2^{-1} + \d  + \s\big) 
+ c \big(h c_\s + \d + h^2 \veps_1^{-1} + (\s^2+ h^2 c_\s^2)\veps_2^{-1} \big).
\end{split}\]
The conditions of the proposition imply that the right-hand tends
to zero as $h\to 0$.  \\
(ii) Assume that $(y_h,b_h)_{h>0}$ is a sequence of admissible discrete
pairs with weak limit
$(y,b)\in V$. We may assume
that the sequence $E_{\d,\veps,h}[y_h,b_h]$ is bounded. Compactness properties
of embeddings imply that $|y'|^2= 1$ and $|b|^2 =1$
in $(0,L)$. Since $P_\veps[y_h,b_h]$ is bounded we have 
that $y_h'\cdot b_h$ and $y_h'\cdot b_h'$
converge to zero in $L^2(0,L)$. Noting that $y_h'\to y$ 
in $L^\infty(0,L;\R^3)$ it follows that $y'\cdot b=0$
and $y'\cdot b' = 0$. Since $L_\bc$ is bounded and linear it is weakly continuous
and we verify
that $L_\bc[y,b] = \ell_\bc$. Hence, we conclude that $(y,b) \in \cA_\frame$.
Since the functional~$E$ is weakly lower semicontinuous, since the functionals 
$P_\veps$ are
nonnegative, and since the regularization of the integrand $\oQ$ is uniformly 
controlled by Lemma~\ref{la:cont_reg}, we find that 
\[\begin{split}
E[y,b] & \le \liminf_{h\to 0} E[y_h,b_h] \\
&\le \liminf_{(\veps,h,\d)\to 0} E_{\d,\veps,h}[y_h,b_h] + c \d \\
&= \liminf_{(\veps,h,\d)\to 0} E_{\d,\veps,h}[y_h,b_h]. 
\end{split}\]
This proves the proposition.
\end{proof}

\begin{remark}
Let $\s=h^\a$, $\veps_j = h^{\b_j}$, $j=1,2$, and $\d = h^\g$
with $\a,\b_1,\b_2,\g>0$. Then the condition of the proposition reads 
\[
h^{2-\b_1} + h^{2\a-\b_2} + h^{2-\b_2}  [y,b]_{h^\a}^2 \to 0.
\]
Assuming that $[y,b]_\s \le c_{{\rm appr}} \s^{-p}$ for some $p\ge 0$
then this condition becomes
\[
h^{2-\b_1} + h^{2\a-\b_2} + h^{2-\b_2 -2 p \a} \to 0.
\]
Hence, for $\b_1<2$ and $\b_2$ sufficiently small the condition is 
satisfied.  If $y\in H^3(0,L;\R^3)$ and
$b\in H^2(0,L;\R^3)$ then we have $p=0$ and the condition is satisfied 
for $\b_2 < \min\{ 2,2\a\}$. 
\end{remark}

Proposition~\ref{prop:gamma_conv} implies the convergence of discrete
almost-minimizers under certain conditions. 

\begin{corollary}[Approximation of regular minimizers]\label{cor:conv_mins}
Let $(y_h,b_h)_{h>0}$ be a sequence of discrete minimizers for the 
problems~\eqref{eq:discr_prob} possibly up to 
tolerances $\veps_{{\rm tol}}$ that vanish as $h\to 0$. If 
an accumulation point $(y,b)\in V$ of the sequence is such that
there exists a sequence $\s\to 0$ such that 
\[
h^2 \veps_1^{-1} + \s^2 \veps_2^{-1} + h^2 [y,b]_\s^2 \veps_2^{-1} \to 0
\]
as $h\to 0$ then every weak accumulation point
$(y,b)\in V$ as $(\d,\veps,h)\to 0$ minimizes the constrained energy
functional~$E$.
\end{corollary}

\begin{proof}
Since the functionals $E_{\d,\veps,h}$ are equi-coercive a sequence
of almost-minimizers is bounded and the result is a direct consequence
of Proposition~\ref{prop:gamma_conv}.
\end{proof}

If the regularization result of Proposition~\ref{prop:strong_approx}
can be refined such that $y_\s'\cdot b_\s'= 0$ for every $\s>0$ then
a $\G$-convergence result can be deduced. 

\begin{corollary}[$\G$-convergence to $E$]\label{cor:gamma_dense}
Assume that for every $\s>0$ there exists a pair $(y_\s,b_\s)$
as in Proposition~\ref{prop:strong_approx} with the additional property
that $y_\s'\cdot b_\s' =0$ in $(0,L)$. If $\veps_1>0$ and
$\veps_2>0$ are chosen such that
\[
h^2 \veps_1^{-1} + h^2 \veps_2^{-1} \to 0
\]
as $h\to 0$ then the constrained functionals $(E_{\d,\veps,h})$ 
are $\G$-convergent to the constrained functional~$E$.
\end{corollary}

\begin{proof}
The functionals $E_{\d,\veps,h}$ are equi-coercive and the condition
of Proposition~\ref{prop:gamma_conv} reduces to the condition 
$h^2 \veps_1^{-1} + h^2 [y,b]_\s^2 \veps_2^{-1} \to 0$ as $h\to 0$.
\end{proof}

A general approximation result can be established if 
the penalty parameter $\veps_2>0$ is fixed so that the corrected
Sadowsky functional with penalized treatment of the orthogonality
$y'\cdot b'=0$ is approximated.  

\begin{corollary}[$\G$-convergence to a penalized approximation]\label{cor:gamma_pen}
Let $\veps_2>0$ be fixed. Then the constrained functionals 
$(E_{\d,\veps,h})$ are $\G$-convergent to the constrained
functional $E_{0,\veps_2}$ as $(\d,\veps_1,h)\to 0$ provided that
$h^2 \veps_1^{-1} \to 0$.
\end{corollary}

\begin{proof}
The result is an immediate consequence of the proof of 
Proposition~\ref{prop:gamma_conv} and the equi-coercivity
of the functionals $E_{\d,\veps,h}$.
\end{proof}

The numerical realization of the functionals $E_{\d,\veps,h}$ requires
the use of quadrature. For an element $T=[x_1,x_2]$ and a 
function $r\in C^1(T)$ we use averaging operators
\[\begin{split}
\cA_h r'|_T &= |T|^{-1} \int_T r' \dv{x} = h_T^{-1} \big(r(x_2)-r(x_1)\big), \\
\cM_h r|_T &= |T|^{-1} \int_T \cI_h^{1,0} r \dv{x} =  \frac12 \big( r(x_1) + r(x_2)\big).
\end{split}\]
We note that we have the estimates 
\[\begin{split}
\|y_h'- \cM_h y_h'\|_{L^\infty(0,L)} &\le c h^{1/2} \|y_h''\|, \\
\|\cA_h y_h'' - y_h''\| & \le c h \|y_h'''\|.
\end{split}\]
Moreover, the operator $\cA_h:L^2(0,L)\to L^2(0,L)$ 
is bounded and self-adjoint; as a consequence of Jensen's or H\"older's 
inequality we have
\begin{equation}\label{eq:jensen}
\|\cA_h y_h'' \| \le \|y_h''\|.
\end{equation}
With the help of the operators $\cA_h$ and $\cM_h$ we can define a 
discrete functional which allows for exact numerical integration and
which $\G$-converges to the continuous variational problem under slightly
more restrictive conditions on the parameters.

\begin{proposition}[Quadrature]\label{prop:quadrature}
The results of Proposition~\ref{prop:gamma_conv} and {Corollaries~\ref{cor:conv_mins},
\ref{cor:gamma_dense}, and~\ref{cor:gamma_pen}}
remain valid if $E_{\d,\veps,h}$ is replaced by the functional 
\[\begin{split}
E_{\d,\veps,h}^{{\rm quad}}[y_h,b_h] & = 
\frac12 \int_0^L |y_h''|^2 + 5 |b_h'|^2 + \psi_\d(|\cA_h y_h''|^2,|b_h'|^2) \dv{x} \\
& \qquad + \frac{1}{2\veps_1} \int_0^L \cI_h^{1,0} \big[(y_h'\cdot b_h)^2\big] \dv{x}
 + \frac{1}{2\veps_2} \int_0^L  (\cM_h y_h'\cdot b_h')^2 \dv{x},
\end{split}\]
i.e., it approximates $E$ in the sense of Proposition~\ref{prop:gamma_conv}
if the additional condition $h\veps_1^{-1} + h^{1/2} \veps_2^{-1/2}$ is 
satisfied. 
\end{proposition}

\begin{proof}
We discuss necessary changes in the proof of Proposition~\ref{prop:gamma_conv}. \\
The penalty terms involving the operators $\cI_h^{1,0}$ and $\cM_h$
can be replaced by the penalty terms without these operators leading to 
error terms that vanish as $h\to 0$, 
i.e., for a bounded sequence $(y_h,b_h)_{h>0}$
in $H^2(0,L;\R^3)\times H^1(0,L;\R^3)$ we have for the first penalty term that 
\[\begin{split}
\Big|\frac{1}{\veps_1}\int_0^L   \cI_h^{1,0} \big[(y_h'\cdot b_h)^2\big] -  & (y_h'\cdot b_h)^2 \dv{x} \Big| 
\le \veps_1^{-1} \|(y_h'\cdot b_h)^2 - \cI_h^{1,0} (y_h' \cdot b_h)^2\|_{L^1(0,L)}  \\
&\le c h \veps_1^{-1} \|(y_h'\cdot b_h) (y_h''\cdot b_h + y_h'\cdot b_h')\|_{L^1(0,L)} \\
&\le c h \veps_1^{-1} \|y_h'\cdot b_h\| \big(\|y_h''\| \|b_h\|_{L^\infty(0,L)} 
  + \|y_h'\|_{L^\infty(0,L)} \|b_h'\|\big) \\
&\le c h \veps_1^{-1},
\end{split}\]
while for the second penalty term we have
\[\begin{split}
\Big|\frac{1}{\veps_2}\int_0^L  & (\cM_h y_h'\cdot b_h')^2 - (y_h'\cdot b_h')^2 \dv{x} \Big| \\
&\le \veps_2^{-1} \|\cM_h y_h' - y_h'\|_{L^\infty(0,L)} \|b_h'\|^2 \|\cM_h y_h' + y_h'\|_{L^\infty(0,L)}  
\le c h^{1/2} \veps_2^{-1}.
\end{split}\]
(i) For the attainment result we note that 
the Lipschitz continuity of $\psi_\d$ (see Lemma~\ref{la:props_psi} below)
yields the estimate
\[\begin{split}
\int_0^L & \Big|\psi_\d(|\cA_h y_h''|^2,|b_h'|^2) - \psi_\d(|y_h''|^2,|b_h'|^2)\Big| \dv{x} \\
&\le c_\psi \int_0^L |\cA_h y_h'' - y_h''||\cA_h y_h''+ y_h''| \dv{x}
\le c h \|y_h'''\| \le c h c_\s,
\end{split}\]
where we used that that the interpolation operator $\cI_h^{3,1}$ is elementwise 
stable in~$H^3$. Hence, we have that
\[
\big|E_{\d,\veps,h}[y_h,b_h]-E_{\d,\veps,h}^{{\rm quad}}[y_h,b_h]\big| \to 0.
\]
(ii) To establish the lower bound property we note that if
 $y_h \wto y$ in $H^2(0,L;\R^3)$ and $b_h\wto b$ in $H^1(0,L;\R^3)$
then~\eqref{eq:jensen} and the weak lower semicontinuity of $E$ identified
in Lemma~\ref{la:convex_and_diffb} show that 
\[\begin{split}
\liminf_{h,\d \to 0} \frac12 & \int_0^L |y_h''|^2 + 5 |b_h'|^2 + \psi_\d(|\cA_h y_h''|^2,|b_h'|^2) \dv{x} \\
&\ge \liminf_{h,\d \to 0} \int_0^L |\cA_h y_h''|^2 + 5 |b_h'|^2 + \psi_\d(|\cA_h y_h''|^2,|b_h'|^2) \dv{x} \\
&\ge \liminf_{h,\d \to 0} \int_0^L |\cA_h y_h''|^2 + 5 |b_h'|^2 + \psi(|\cA_h y_h''|^2,|b_h'|^2) \dv{x}  - c \d \\
&= \liminf_{h \to 0} \int_0^L \oQ (|\cA_h y_h''|,|b_h'|) \dv{x} \\
&\ge \int_0^L |y''|^2 + 5 |b'|^2 + \psi(|y''|^2,|b'|^2) \dv{x}.
\end{split}\]
This implies that $E[y,b] \le \liminf_{(\veps,\d,h)\to 0} 
E_{\d,\veps,h}^{{\rm quad}}[y_h,b_h]$ and proves the assertion. 
\end{proof}

\begin{remark}
An inspection of the proof of Proposition~\ref{prop:gamma_conv}
implies that no condition on $\veps_1$ is required if quadrature 
is used. For the recovery sequence $(y_h,b_h)$ constructed in~(i)
we have that the nodal values of $y_h'\cdot b_h$ and hence the
penalty term vanish. For a sequence $(y_h,b_h)$ in~(ii) with bounded 
energies we have $\|\cI_h[(y_h'\cdot b_h)^2]\|_{L^1(0,L)} \le c \veps$
which implies that $y'\cdot b=0$ for accumulation points. 
\end{remark}

\section{Iterative minimization}\label{sec:iterative}

\subsection{Discrete gradient flow}
We devise a decoupled gradient scheme with a linearized treatment
of the pointwise constraints to compute stationary configurations. 
For ease of readability we often omit the spatial discretization
parameter $h$ in what follows and denote
\[
Y_h = \cS^{3,1}(\cT_h)^3, \quad B_h = \cS^{1,0}(\cT_h)^3.
\]
The variational derivatives of the (discretized) penalty 
functional $P_{\veps,h}$ with respect to $y$ and $b$
are denoted by $\d_y P_\veps$ and $\d_b P_\veps$, respectively.
For the function
$\psi_\d$ the partial derivatives are denoted by $\psi_{\d,1}' $ and
$\psi_{\d,2}'$. We assume that the boundary conditions are separated,
i.e., that we have
\[
L_\bc[y,b] = \big(L_\bc^1[y],L_\bc^2[b]\big),
\]
and $\ell_\bc = (\ell_\bc^1,\ell_\bc^2)$. 

\begin{algorithm}[Gradient descent]\label{alg:grad_desc}
Choose $\tau>0$ and $(y_0,b_0)\in Y_h\times B_h$ with 
\[
L_\bc^1[y_0] = \ell_\bc^1, \quad L_\bc^2[b_0] = \ell_\bc^2
\]
and the nodal length conditions
\[
|y_0'|^2=1, \quad |b_0|^2 = 1;
\]
set $k=1$. \\
(1) Compute $y_k \in Y_h$ such that 
\[\begin{split}
 (d_t y_k,w)_\star + (y_k'',w'') +  \d_y P_\veps[y_k&,b_{k-1};w] = \\ 
& - (\psi_{\d,1}'(|\cA y_{k-1}''|^2,|b_{k-1}'|^2) \cA y_{k-1}'',\cA w'') 
\end{split}\]
for all $w\in Y_h$ subject to the nodal constraints
\[
d_t y_k'\cdot y_{k-1}' = 0, \quad w'\cdot y_{k-1}' = 0
\]
and the boundary conditions 
\[
L_\bc^1[d_t y_k] = 0, \quad L_\bc^1[w] = 0.
\]
(2) Compute  $b_k \in B_h$ such that 
\[\begin{split}
 (d_t b_k,r)_\dagger + 5 (b_k',r') + \d_b P_\veps[y_k&,b_k;r] = \\
& - (\psi_{\d,2}'(|\cA y_{k-1}''|^2,|b_{k-1}'|^2) b_{k-1}',r') 
\end{split}\]
for all $r\in B_h$ subject to the nodal constraints
\[
d_t b_k \cdot b_{k-1} = 0, \quad r \cdot b_{k-1} = 0
\]
and the boundary conditions 
\[
L_\bc^2[d_t b_k] = 0, \quad L_\bc^2[r] = 0.
\]
(3) Stop the iteration if 
$\|d_t y_k\|_\star + \|d_t b_k\|_\dagger \le \veps_{{\rm stop}}$; 
set $k\to k+1$ and continue with~(1) otherwise.
\end{algorithm}

We remark that all steps of Algorithm~\ref{alg:grad_desc} define linear
systems of equations that can be solved efficiently. 

\subsection{Linearization estimates}
The iterative numerical treatment of the model problem requires the use of
linearizations of nonlinearities occurring in the energy functional. 
The following lemmas provide estimates on derivatives of the
nonquadratic part of the regularized function~$\oQ_\d$. 

\begin{lemma}[Derivative bounds]\label{la:props_psi}
The function $\psi_\d$, for $r,s\in \R_{\ge 0}$ defined via
\[
\psi_\d(r,s) =  |r-s|_\d + \frac{4 s^2}{r+s+|r-s|_\d},
\]
satisfies
\[
|\nabla \psi_\d| \le c_{\psi,1}, \quad |D^2 \psi_\d|\le c_{\psi,2} \d^{-1}.
\]
\end{lemma}

\begin{proof}
Let $\s_\d = |\cdot|_\d'$ and $D_\d = \s_\d'= |\cdot|_\d''$. We have
\[\begin{split}
\p_1 \psi_\d (r,s) &= \s_\d(r-s) - \frac{4s^2}{(r+s+|r-s|_\d)^2}(1+\s_\d(r-s)) \\
\p_2 \psi_\d (r,s) &= - \s_\d(r-s) + \frac{8s}{r+s+|r-s|_\d}  
- \frac{4 s^2}{(r+s+|r-s|_\d)^2}(1-\s_\d(r-s)).
\end{split}\]
The formulas imply the uniform boundedness of the partial derivatives. 
For higher order partial derivatives we have, e.g., 
\begin{multline*}
\p_1^2 \psi_\d (r,s) = D_\d(r-s) + \frac{8s^2}{(r+s+ |r-s|_\d)^3}\big(1+\s_\d(r-s)\big)^2 \\
  - \frac{4s^2}{(r+s+|r-s|_\d)^2}D_\d(r-s).
\end{multline*}
The remaining second order derivatives lead
to similar formulas. Incorporating the assumptions on the
second derivative of the regularized modulus function and using 
a case distinction relating $r+s$ to $\d$ imply the bounds 
on the Hessian of $\psi_\d$. 
\end{proof}

A Taylor formula is needed to control the explicit treatment of the nonlinear
part~$\psi_\d$ in Algorithm~\ref{alg:grad_desc}.

\begin{lemma}[Taylor formula]\label{la:taylor_psi}
For all $a,b,\ta,\tb \in \R^3$ there exist $\ha,\hb \in \R^3$ such that 
\[\begin{split}
\frac12 \big(\psi_\d(|a|^2,|b|^2) - \psi_\d(|\ta|^2,|\tb|^2)& \big) 
- \nabla \psi_\d(|\ta|^2,|\tb|^2) \cdot 
\begin{bmatrix} \ta \cdot (a-\ta)  \\ \tb \cdot (b-\tb) \end{bmatrix} \\
& =  \G(\ha,\hb,a,b,\ta,\tb),
\end{split}\]
with a function $\G$ that satisfies
\[
\big|\G(\ha,\hb,a,b,\ta,\tb)\big| \le 
c_\G \big[\d^{-1} \big(|a|^2 + |\ta|^2 + |b|^2 + |\tb|^2\big) + 1\big]
\big(|a-\ta|^2 + |b-\tb|^2\big).
\]
\end{lemma}

\begin{proof}
Defining $g(a,b) = (1/2) \psi_\d(|a|^2,|b|^2)$ we have
\[
g(a,b) -  g(\ta,\tb) - D g(\ta,\tb) \begin{bmatrix} a-\ta \\ b-\tb \end{bmatrix}
=  \frac12 \begin{bmatrix} a-\ta \\ b-\tb \end{bmatrix}^\transp
 D^2g(\ha,\hb)  \begin{bmatrix} a-\ta \\ b-\tb \end{bmatrix},
\]
where $(\ha,\hb)$ belongs to the line segment connecting $(a,b)$ with 
$(\ta,\tb)$. We have
\[
\p_1 g(\ha,\hb) = \p_1 \psi_\d(|\ha|^2,|\hb|^2) \ha, \quad
\p_2 g(\ha,\hb) = \p_2 \psi_\d(|\ha|^2,|\hb|^2) \hb,
\]
and consequently, 
\[\begin{split}
\p_1^2 g(\ha,\hb) &= 2\p_1^2 \psi_\d(|\ha|^2,|\hb|^2) \ha \ha^\transp  + \p_1 \psi_\d(|\ha|^2,|\hb|^2) I_{3\times 3}, \\
\p_2^2 g(\ha,\hb) &= 2 \p_2^2 \psi_\d(|\ha|^2,|\hb|^2) \hb \hb^\transp + \p_2 \psi_\d(|\ha|^2,|\hb|^2) I_{3\times 3}, \\
\p_1 \p_2 g(\ha,\hb) &= 2\p_1 \p_2 \psi_\d(|\ha|^2,|\hb|^2) \hb \, \ha^\transp.
\end{split}\]
The bounds of Lemma~\ref{la:props_psi} imply the asserted result.
\end{proof}

\subsection{Energy monotonicity}
Stability of Algorithm~\ref{alg:grad_desc} can be established 
if the step size~$\tau$
satisfies a smallness condition defined by the regularization 
parameter~$\d$ and the mesh-size~$h$. The following proposition
implies that the iteration approximates stationary points for
the discrete energy functional. 

\begin{proposition}[Energy decay]\label{prop:ener_decay}
Assume that there exist $c_\star,c_\dagger >0$ such that for all
$w\in Y_h$ and $r \in B_h$ we have
\[
\|w''\| \le c_\star \|w\|_\star, \quad \|r'\| \le c_\dagger \|r\|_\dagger.
\]
Then there exists a constant $c_{{\rm ed}}>0$ such that
for the uniquely defined iterates $(y_k,b_k)_{k=0,1,\dots}$ of 
Algorithm~\ref{alg:grad_desc} we have  
\[
E_{\d,\veps,h}[y_K,b_K] 
+ \big(1-c_{{\rm ed}} \tau \d^{-1}h^{-1}\big) 
\tau \sum_{k=1}^K \big(\|d_t y_k\|_\star^2 + \|d_t b_k\|_\dagger^2 \big) 
\le E_{\d,\veps,h}[y_0,b_0].
\]
Moreover, if $2 c_{{\rm ed}} \tau \le \d h$ we have that
\[
\|\cI_h^{1,0} |b_k|^2-1\|_{L^\infty(0,L)} + \|\cI_h^{1,0}|y_k'|^2 -1\|_{L^\infty(0,L)} 
\le  c_{{\rm ul}} \tau e_0,
\]
with $e_0 = E_{\d,\veps,h}[y_0,b_0]$.
\end{proposition}

\begin{proof}
We adopt an inductive argument and first prove an intermediate bound which 
will allow us to control the explicit treatment of the nonlinearities related
to $\psi$. The conditional energy decay estimate then implies the control of
the violation of the pointwise constraints. Well-posedness of the iteration
is an immediate consequence of the Lax--Milgram lemma.  \\
{\em Step 1.} We choose $w=d_t y_k$ and $r= d_t b_k$ in Steps~(1) and~(2) of 
Algorithm~\ref{alg:grad_desc}. Incorporating~\eqref{eq:jensen} this implies that 
\begin{multline}\label{eq:rough_est}
\|d_t y_k\|_\star^2 + \|d_t b_k\|_\dagger^2 + 
d_t \Big(\frac12 \|y_k''\|^2 + \frac52 \|b_k'\|^2 + P_{\veps,h}[y_k,b_k] \Big) \\
\le \|\nabla \psi_\d(|\cA y_{k-1}''|^2,|b_{k-1}'|^2)\|_{L^\infty(0,L)} 
\big(\|y_{k-1}''\| \| d_t y_k''\| + \|b_{k-1}'\| \|d_t b_k''\| \big).
\end{multline}
Here, we used that the separate convexity of $P_{\veps,h}$ implies that we have
\[
P_{\veps,h}[y_k,b_{k-1}] + \d_y P_{\veps,h}[y_k,b_{k-1};y_{k-1}-y_k] 
\le P_{\veps,h}[y_{k-1},b_{k-1}]
\]
and 
\[
P_{\veps,h}[y_k,b_k] + \d_b P_{\veps,h}[y_k,b_k;b_{k-1}-b_k] 
\le P_{\veps,h}[y_k,b_{k-1}]
\]
which by summation and multiplication by $1/\tau$ yields that
\[\begin{split}
d_t P_{\veps,h}[y_k,b_k] &= \frac{1}{\tau}\big(P_{\veps,h}[y_k,b_k]-P_{\veps,h}[y_{k-1},b_{k-1}]\big) \\
&\le \d_y P_{\veps,h} [y_k,b_{k-1};d_t y_k] + \d_b P_{\veps,h}[y_k,b_k;d_t b_k].
\end{split}\]
In an inductive argument we assume that $E_{\d,\veps}[y_{k-1},b_{k-1}] \le e_0$
which yields
\[
\frac12 \|y_{k-1}''\|^2 + \frac52 \|b_{k-1}'\|^2 + P_{\veps,h}[y_{k-1},b_{k-1}] \le e_0.
\]
With~\eqref{eq:rough_est} and Lemma~\ref{la:props_psi} we find that
\[\begin{split}
\tau  \Big(\frac12 \|d_t y_k\|_\star^2 & + \frac12 \|d_t b_k\|_\dagger^2\Big) 
+ \Big(\frac12 \|y_k''\|^2 + \frac52 \|b_k'\|^2 + P_{\veps,h}[y_k,b_k] \Big) \\
& \le e_0 + \tau \frac{c_{\psi,1}^2}{2} \max\{c_\star^2,c_\dagger^2\} \big(\|y_{k-1}''\|^2 + \|b_{k-1}'\|^2\big) \\
& \le e_0 + \tau \frac{c_{\psi,1}^2}{2} \max\{c_\star^2,c_\dagger^2\} 2 e_0 
= (1+\tau C_0) e_0 \le 2 e_0,
\end{split}\]
where we assumed that $\tau C_0 \le 1$.  The next step shows that $E_{\d,\veps,h}[y^k,b^k]\le e_0$
and hence implies that the right-hand side can be replaced by~$e_0$. \\
{\em Step 2.} We again choose $w=d_t y_k$ and $r= d_t b_k$ in Steps~(1) and~(2) of 
Algorithm~\ref{alg:grad_desc} but use the Taylor formula of Lemma~\ref{la:taylor_psi}
to verify that 
\begin{multline*}
d_t \Big(\frac12 \|y_k''\|^2 + \frac52 \|b_k'\|^2 
+ \int_0^L \psi_\d(|\cA y_k''|^2,|b_k'|^2) \dv{x}
+ P_{\veps,h}[y_k,b_k] \Big) \\
+ \|d_t y_k\|_\star^2 + \|d_t b_k\|_\dagger^2 
\le \int_0^L \G(\hy_k,\hb_k,y_{k-1},b_{k-1},y_k,b_k) \dv{x},
\end{multline*}
where, using that $\|\cA r\|_{L^\infty(0,L)} \le \|r\|_{L^\infty(0,L)}$, 
\begin{multline*}
\int_0^L \G(\hy_k,\hb_k,y_{k-1},b_{k-1},y_k,b_k) \dv{x} \\
\le c_\G \tau \d^{-1} \big(\|y_k''\|_{L^\infty(0,L)}^2 + \|y_{k-1}''\|_{L^\infty(0,L)}^2 +  
\|b_k'\|_{L^\infty(0,L)}^2 + \|b_{k-1}'\|_{L^\infty(0,L)}^2+1\big) \\
\times  \big(\|d_t y_k''\|^2 + \|d_t b_k'\|^2\big). 
\end{multline*}
Using the available bounds and the inverse estimate~\eqref{eq:inv_est},
we deduce that 
\[
\int_0^L \G(\dots) \dv{x} \le c_\G \tau \d^{-1} \big(4 c_\inv h^{-1} 8 e_0+1\big)
\max\{c_\star^2,c_\dagger^2\} \big(\|d_t y_k\|_\star^2 + \|d_t b_k\|_\dagger^2\big).
\]
This implies that we have 
\[
d_t E_{\d,\veps}[y_k,b_k] + \big(1- c_{{\rm ed}} \tau \d^{-1} h^{-1}\big) 
\big(\|d_t y_k\|_\star^2 + \|d_t b_k\|_\dagger^2\big) \le 0.
\]
A summation over $k=1,2,\dots,K$ and multiplication by~$\tau$
yield the energy estimate for some~$K$, which implies the induction hypothesis. \\
{\em Step 3.} To derive the bounds that control the violation of the constraints 
we note that at the nodes we have 
\[
|b_k|^2 = |b_{k-1}|^2 + \tau^2 |d_t b_k|^2 
= \dots = |b_0|^2 + \tau^2 \sum_{r=1}^k |d_t b_k|^2.
\]
Incorporating the identity $|b_0|^2 =1$ and taking the $L^\infty$ norm implies the 
constraint violation estimate.  The same argument applies to $|y_k'|^2$. 
\end{proof}

\section{Numerical Experiments}\label{sec:num_ex}

\subsection{Realization}
With the quadrature introduced in Proposition~\ref{prop:quadrature}, 
Algorithm~\ref{alg:grad_desc}
can be realized exactly and only requires solving linear systems of 
equations. These define certain linear saddle-point problems with 
unique solutions. We always use the inner products related to the 
norms 
\[
\|w\|_\star^2 = \|w\|^2 + \|w''\|^2, \quad \|r\|_\dagger^2 = \|r\|^2 + \|r'\|^2,
\]
and the relations
\[
\tau = h/10, \quad \veps_1 = h, \quad \veps_2 = h^{1/2}, \quad \d = h^{1/2}.
\]
In view of Proposition~\ref{prop:ener_decay} this choice for~$\tau$ may
be too optimistic in general. However, we always observed energy monotonicity
in our experiments which may be related to an additional regularity property.
We confirm our theoretical findings by showing that our numerical scheme 
reliably and efficiently detects 
stationary configurations of low energy via discrete evolutions from given 
initial states. In what follows we visualize the ribbon by introducing
an artificial positive width. In addition to the discretized dimensionally reduced 
eleastic energy $E_{\d,\veps,h}$ and the penalty term $P_{\veps,h}$
we also investigate the behavior of the bending and twist energies given by
\[
E_{{\rm bend}}[y] = \frac12 \int_0^L |y''|^2 \dv{x}, \quad
E_{{\rm twist}}[b] = \frac12 \int_0^L |b'|^2 \dv{x}.
\]
We always consider clamped boundary conditions for the deformation~$y$ and
Dirichlet boundary conditions for the director~$b$ which are specified
via the initial states, i.e., we impose 
\[
y(x) = y_0(x), \quad y'(x) = y_0'(x), \quad b(x) = b_0(x).
\]
at the endpoints $x\in \{0,L\}$. To satisfy the condition $y_0'\cdot b_0 = 0$
the vector field $b_0$ is obtained by rotating a non-tangential vector field~$n_0$
around the tangent vector $y_0'$. The condition $y_0'\cdot b_0' =0$ is not
taken into account in the construction of $b_0$ but will be approximated via
the involved penalization.
The discrete initial states are defined
via nodal interpolation of the continuous initial states, i.e., 
\[
y_h^0 = \cI_h^{3,1} y_0, \quad b_h^0 = \cI_h^{1,0} b_0.
\]
All computations are carried out on triangulations that are given by
uniform partitions of the intervals $(0,L)$ into $N$ subintervals of 
length $h=L/N$. We occasionally refer to the quantities $|y_h''|$ and
$|b_h'|$ as curvature and torsion of a ribbon. We always ran the discrete
gradient flow with the fixed time horizon $T=10$. 

\subsection{M\"obius ribbon}
We consider boundary conditions that lead to the formation of a M\"obius 
strip of small, vanishing width. The following example defines a twisted
frame for the unit circle which will deform into a stationary configuration
for the reduced elastic energy. 

\begin{example}[M\"obius ribbon]\label{ex:moebius}
Set $L=2 \pi$ and define for $x\in (0,L)$ 
\[
y_0(x) = \big( \cos(x),\sin(x), 0\big)
\]
and with $\a= 3/2$ 
\[
b_0 = \cos(\a \cdot) (y_0' \times n_0) \times y_0' 
+ \sin(\a \cdot) y_0' \times n_0 ,
\]
with $n_0 = [y_0']_2^\perp$ being the rotation of the first two components
of $y_0'$ by $\pi/2$, i.e., $[a]_2^\perp = (-a_2,a_1,0)$ for $a\in \R^3$. 
\end{example}

Figure~\ref{fig:band_snaps_exp_1} shows snapshots of the discrete evolution
defined by Algorithm~\ref{alg:grad_desc} in Example~\ref{ex:moebius} for a partitioning
of $(0,L)$ into $N=320$ intervals. We observe an immediate change in shape
of the ribbon towards a M\"obius strip. The initially constant curvature
becomes nonconstant and appears to vanish in neighborhoods of three points 
in the nearly stationary configuration. This effect can also be seen in 
Figure~\ref{fig:band_finals_exp_1} where the stationary state is visualized
from two other perspectives. In points where curvature vanishes 
the original Sadowsky functional is singular and the correction from~\cite{FHMP16}
becomes particularly relevant. The pure bending energy defined by the deformation 
of the centerline
increases during the discrete gradient flow evolution, while the twist energy
and the dimensionally reduced energy decay monotonically and quickly become
nearly stationary, cf.~Figure~\ref{fig:ener_exp_1}. The discrete director
fields $b_h^0$ and $b_h^k$ with $k=5092$ are displayed at selected deformed
nodes in Figure~\ref{fig:directors_exp_1}. An experimental convergence
behavior of the stationary energies 
for discretizations with $h=L/N$, $N=80, 160, 320, 640$ is provided
in Table~\ref{tab:conv_eners_exp_1}. We observe that the energies increase
with finer resolutions and their differences become smaller. These results also
indicate that the scheme does not tend to get stuck at local minima. 

\begin{table}
\begin{tabular}{|l|c|c|c|c|} \hline
$N\sim h^{-1}$ & 80 & 160 & 320 & 640 \\\hline
$E^{{\rm quad}}_{\d,\veps,h}[y_h,b_h]$ & 14.9255 & 32.0886 & 35.5842 & 37.6865 \\ \hline 
\end{tabular}
\vspace*{2mm}
\caption{\label{tab:conv_eners_exp_1} Stationary energy values in 
Example~\ref{ex:moebius} for different mesh-sizes $h=L/N$.}
\end{table}

\begin{figure}[p]
\includegraphics[width=6.5cm]{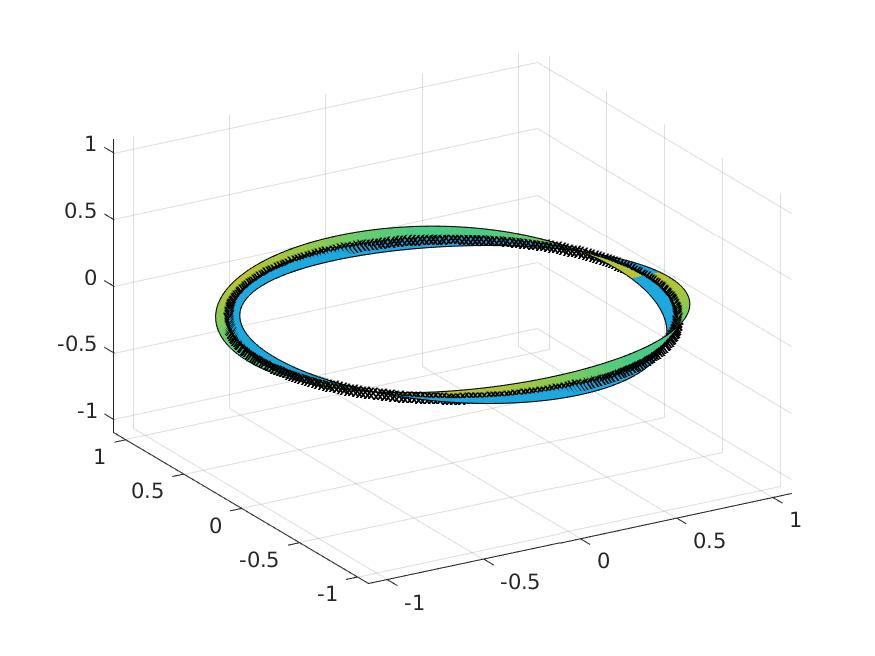} \hspace*{-8mm}
\includegraphics[width=6.5cm]{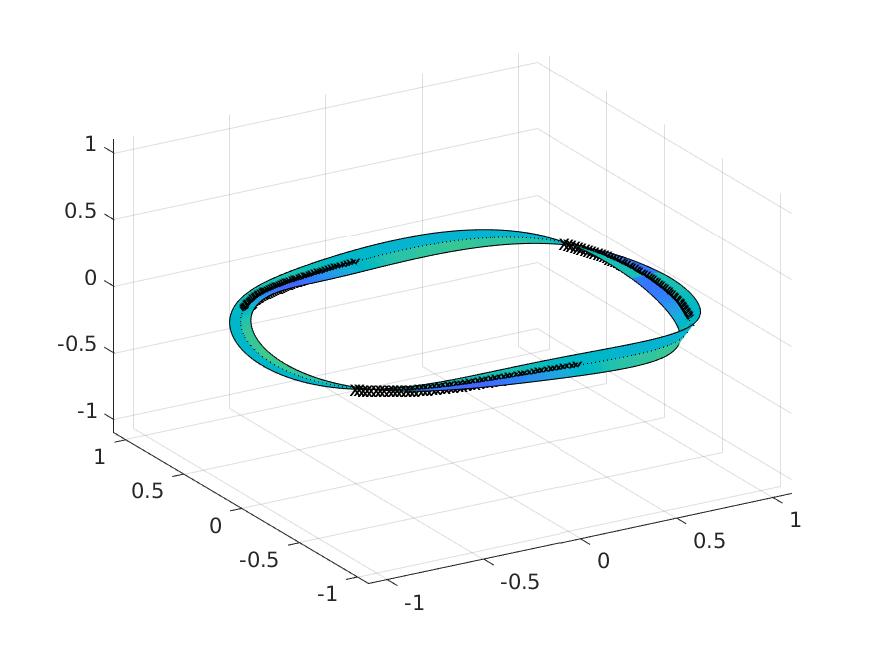} \\[-3mm]
\includegraphics[width=6.5cm]{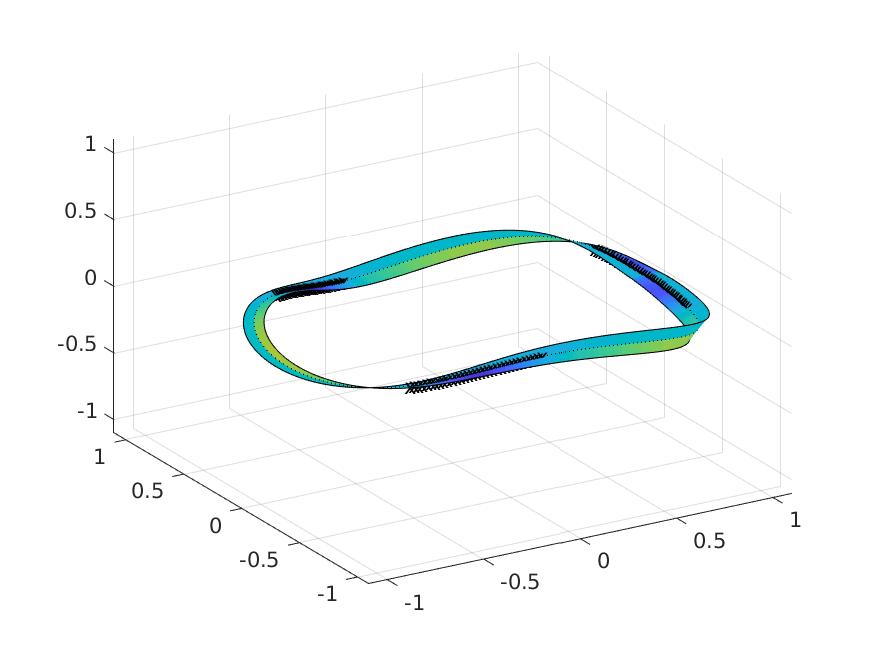}  \hspace*{-8mm}
\includegraphics[width=6.5cm]{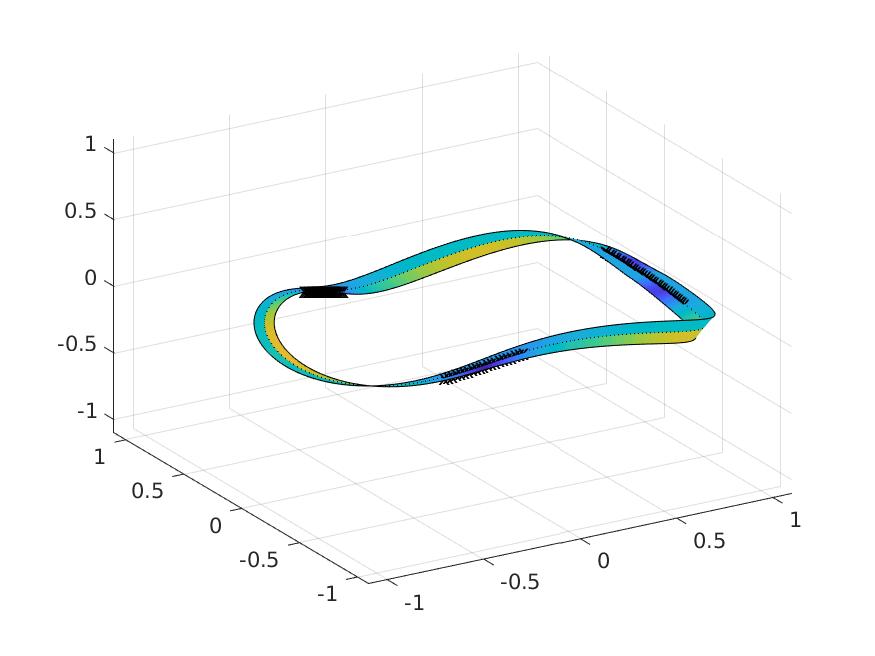} \\[-3mm]
\includegraphics[width=6.5cm]{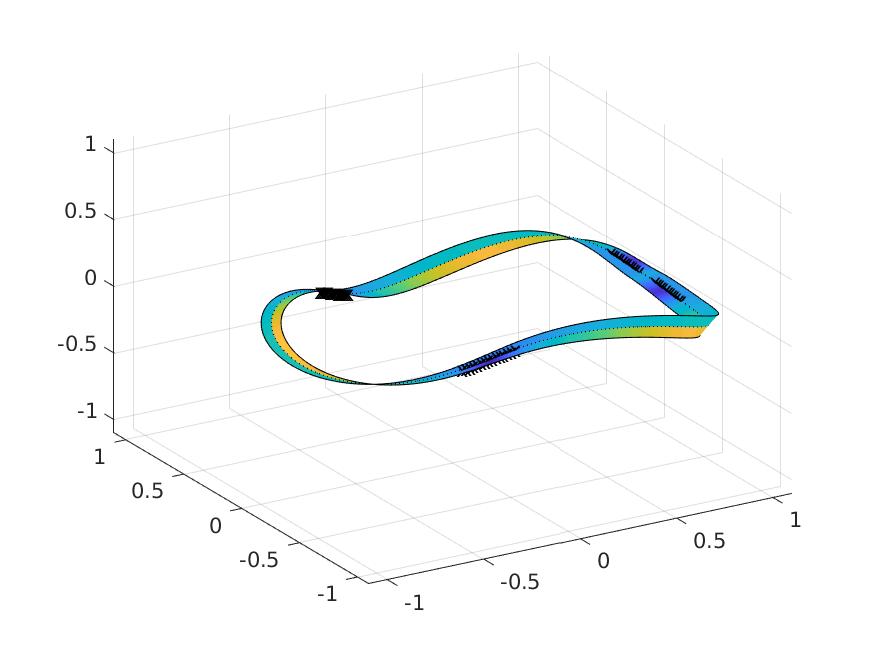} \hspace*{-8mm} 
\includegraphics[width=6.5cm]{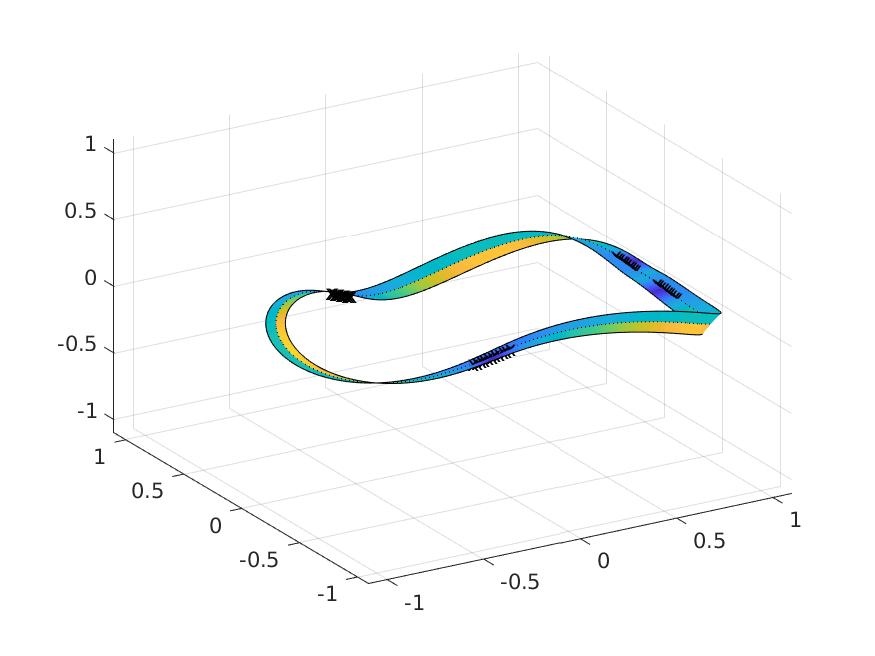} \\[-3mm]
\includegraphics[width=6.5cm]{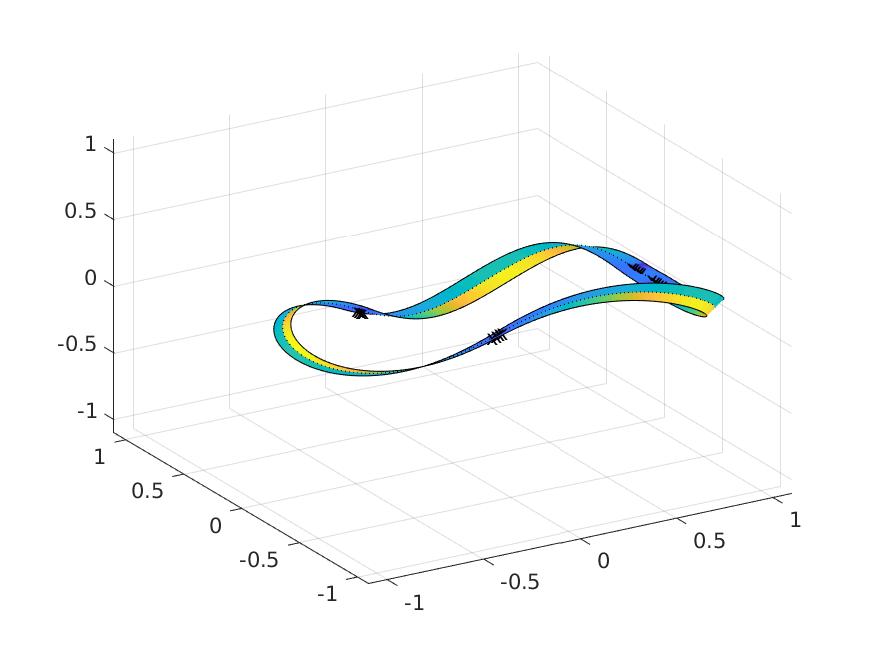} \hspace*{-8mm} 
\includegraphics[width=6.5cm]{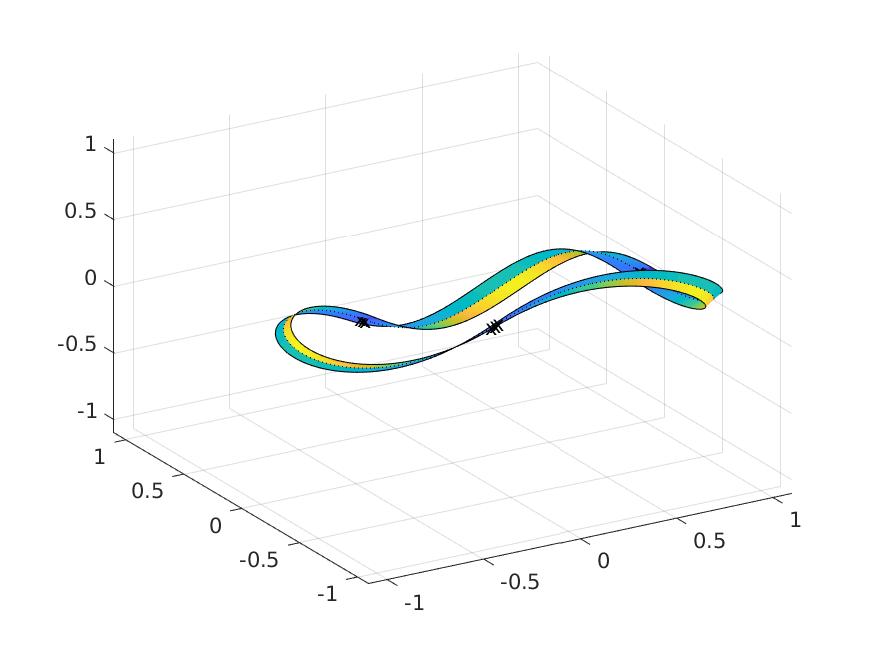}
\caption{\label{fig:band_snaps_exp_1} Snapshots of the discrete gradient flow
in Example~\ref{ex:moebius} after $k=0, 102, 204, 306, 408, 510, 2448, 5092$ iterations
leading to a M\"obius strip, colored by curvature and torsion; 
crosses indicate where torsion dominates curvature.}
\end{figure}

\begin{figure}[p]
\includegraphics[width=5.6cm]{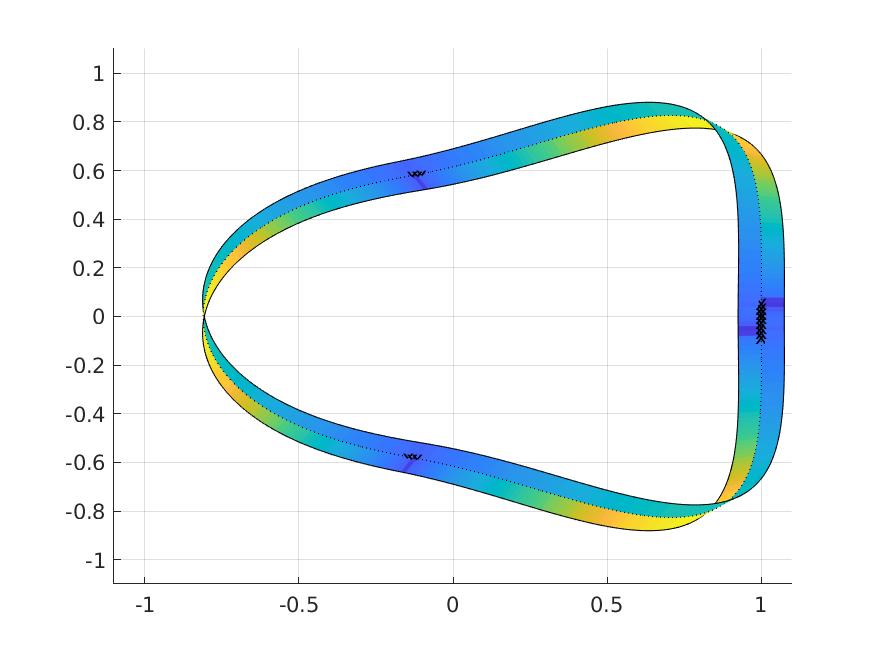}  \hspace*{-3mm}
\includegraphics[width=5.6cm]{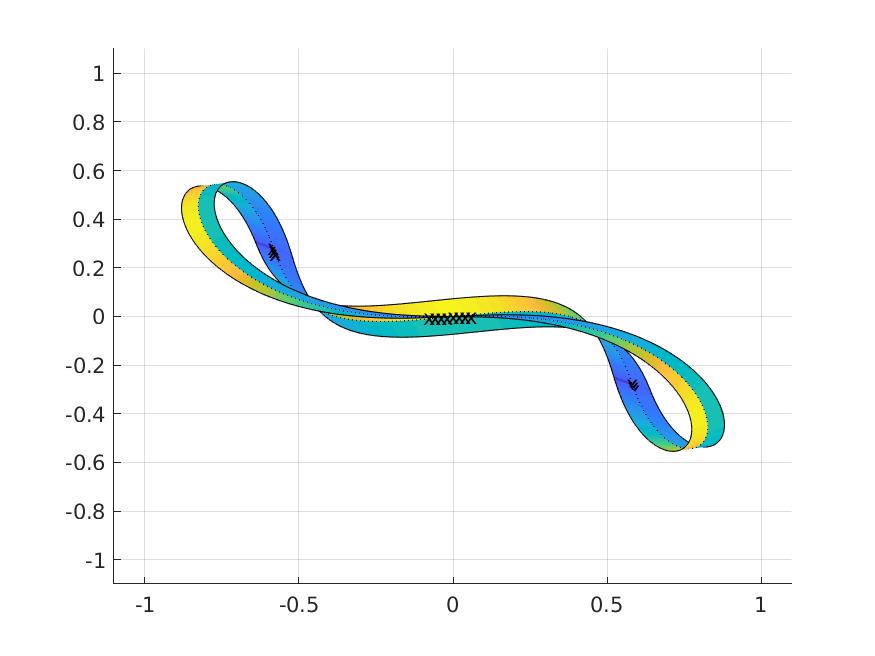} \\
\caption{\label{fig:band_finals_exp_1} Nearly stationary configuration in 
Example~\ref{ex:moebius} from different perspectives colored by curvature and
torsion; crosses indicate where torsion dominates curvature.}
\end{figure}

\begin{figure}
\includegraphics[width=5.8cm]{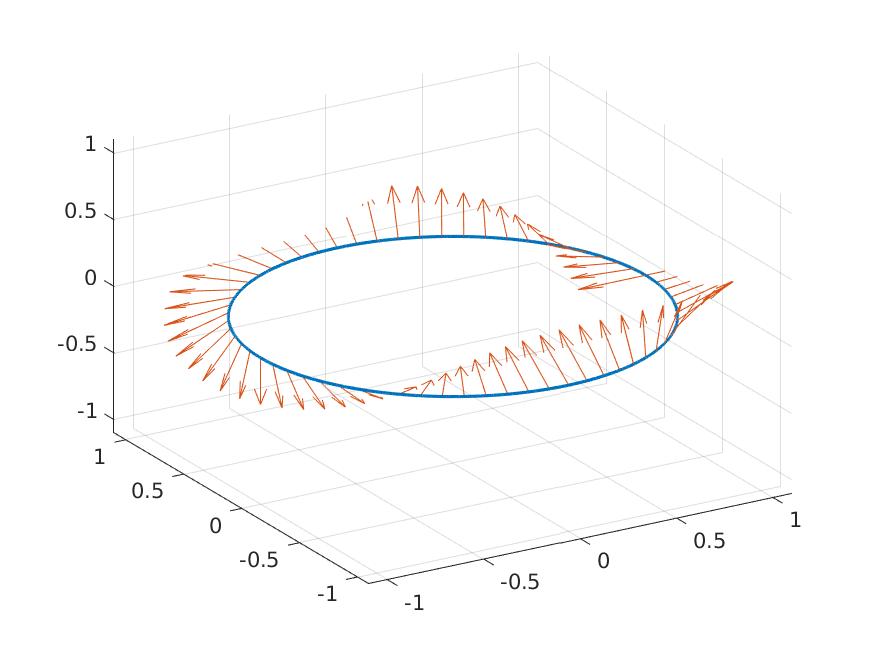} 
\includegraphics[width=5.8cm]{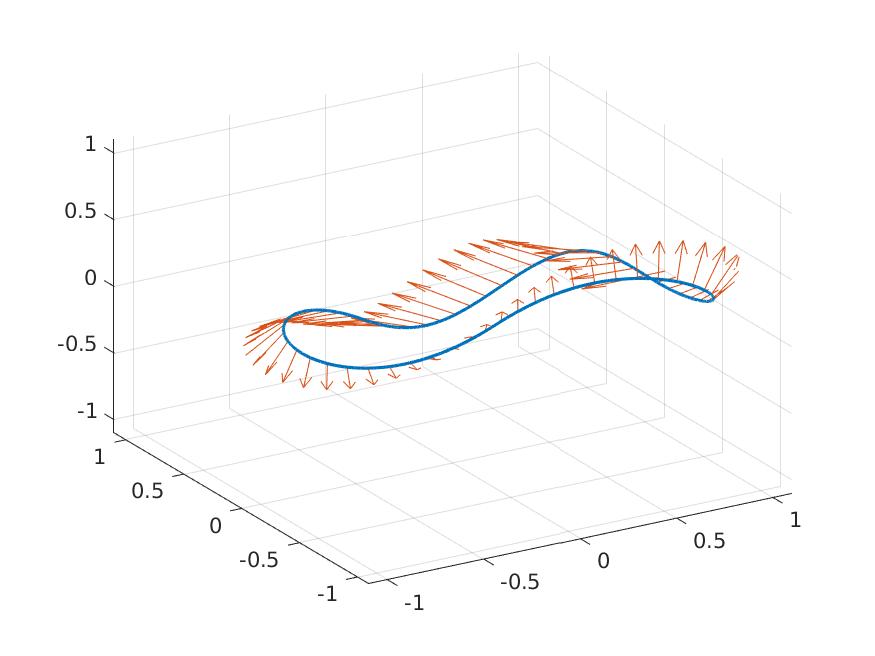} \\
\caption{\label{fig:directors_exp_1} Initial and nearly stationary director
fields on the deformed centerline 
in Example~\ref{ex:moebius}; the vectors at every fourth node are displayed.}
\end{figure}

\begin{figure}
\includegraphics[width=7.7cm]{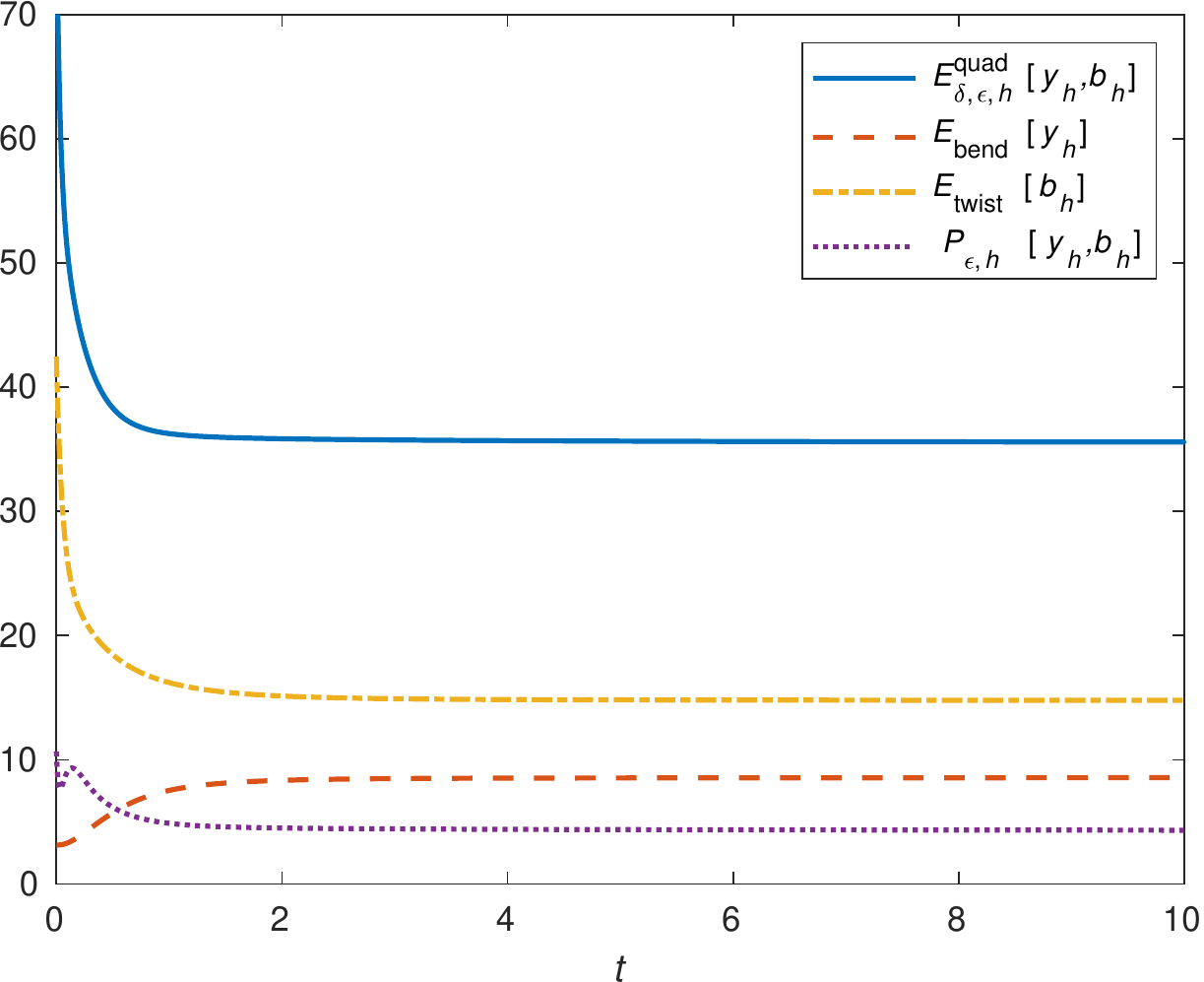} 
\caption{\label{fig:ener_exp_1} Energy decay in Example~\ref{ex:moebius}
accompanied by an increase of bending energy; twist energy and penalty functional
decay monotonically.}
\end{figure}

\subsection{Twisted helix}
A helix is a canonical example of a twisted band. We consider here the
evolution from a helical configuration where the torsion is inconsistent
with the winding of the helix.

\begin{example}[Twisted helical ribbon]\label{ex:helix}
Let $L = 2\pi$ and define for $x\in (0,L)$
\[
y_0(x) = \big(c x, d  \cos(\b x), d \sin(\b x) \big),
\]
where $c=0.95$ and $d= (1-c^2)^{1/2}/\b$ with $\b=2$, and
with $\a=1$ let
\[
b_0 = \Pi_{S_1}\big[\cos(\a \cdot) (y_0' \times n_0) \times y_0' 
+ \sin(\a \cdot ) y_0' \times n_0\big] ,
\]
where $n_0 = [y_0']_2^\perp$, cf.~Example~\ref{ex:moebius}, and
$\Pi_{S_1}:\R^3\setminus \{0\}\to \R^3$ denotes the projection 
onto the unit sphere. 
\end{example}

The application of Algorithm~\ref{alg:grad_desc} to the initial
data defined by Example~\ref{ex:helix} leads to the snapshots 
displayed in Figure~\ref{fig:band_snaps_exp_2}. In contrast
to Example~\ref{ex:moebius} we observe in Example~\ref{ex:helix}
that in the nearly stationary configuration curvature is dominated
by torsion along the entire centerline, cf.~Figure~\ref{fig:band_finals_exp_2}.
Interestingly, this is also the case for the inital configuration but
changes during the evolution. The discrete, dimensionally reduced energy,
the twist energy, and the penalty value decrease montonically during
the evolution while the bending energy shows a moderate increase as 
can be seen in the plot of Figure~\ref{fig:ener_exp_2}. The initial
and nearly stationary director fields are displayed in 
Figure~\ref{fig:directors_exp_2}. The stationary energies for a 
sequence of refined partitions are provided in Table~\ref{tab:conv_eners_exp_2}.
In comparison with Example~\ref{ex:moebius} we observe here a non-monotone
behavior and smaller differences. 

\begin{table}
\begin{tabular}{|l|c|c|c|c|} \hline
$N\sim h^{-1}$ & 80 & 160 & 320 & 640 \\\hline
$E^{{\rm quad}}_{\d,\veps,h}[y_h,b_h]$ & 27.7554 & 26.5432 & 26.3554 & 26.4050\\ \hline 
\end{tabular}
\vspace*{2mm}
\caption{\label{tab:conv_eners_exp_2} Stationary energy values in 
Example~\ref{ex:helix} for different mesh-sizes $h=L/N$.}
\end{table}

\begin{figure}[p]
\includegraphics[width=6.5cm]{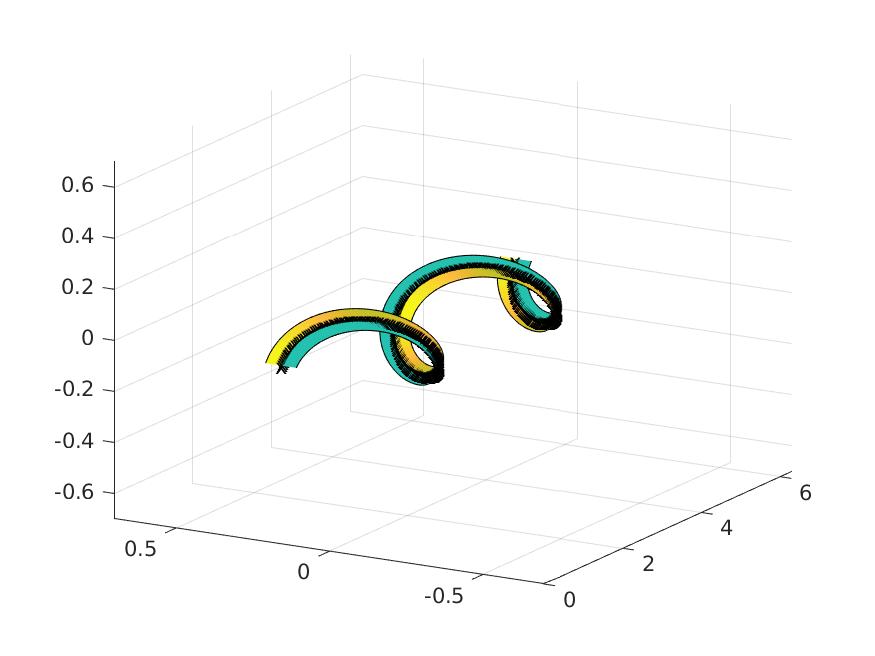} \hspace*{-8mm}
\includegraphics[width=6.5cm]{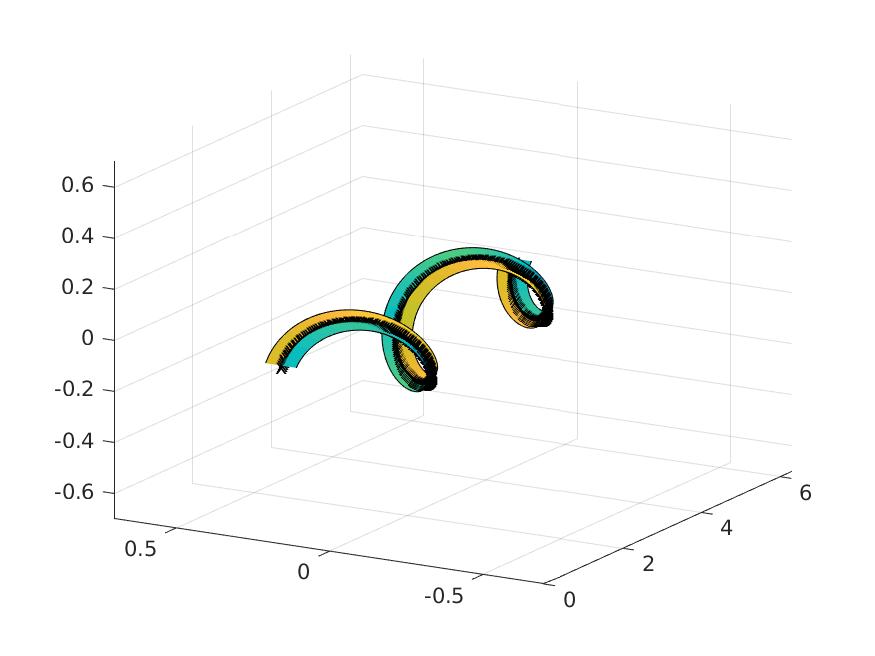} \\[-2mm]
\includegraphics[width=6.5cm]{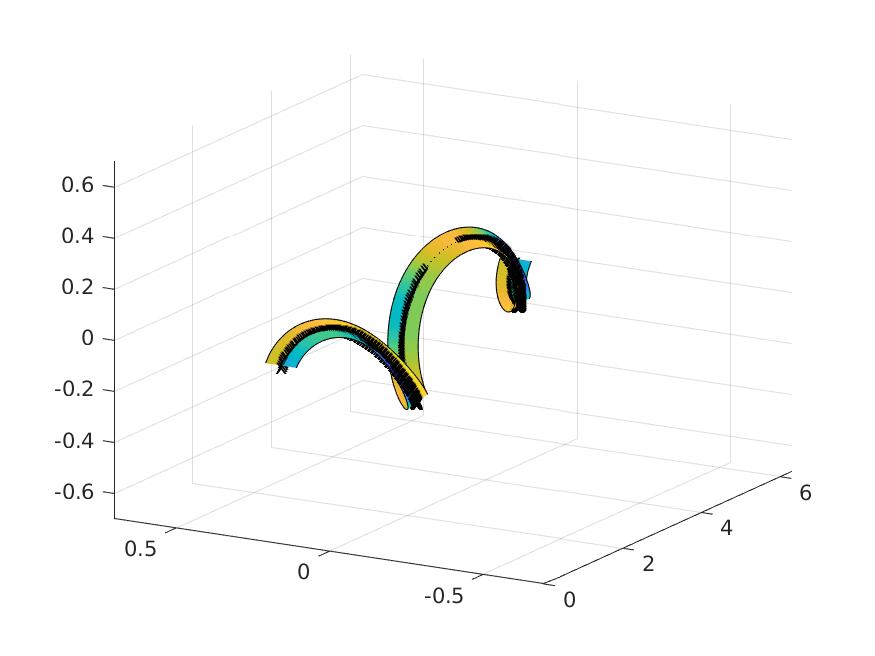}  \hspace*{-8mm}
\includegraphics[width=6.5cm]{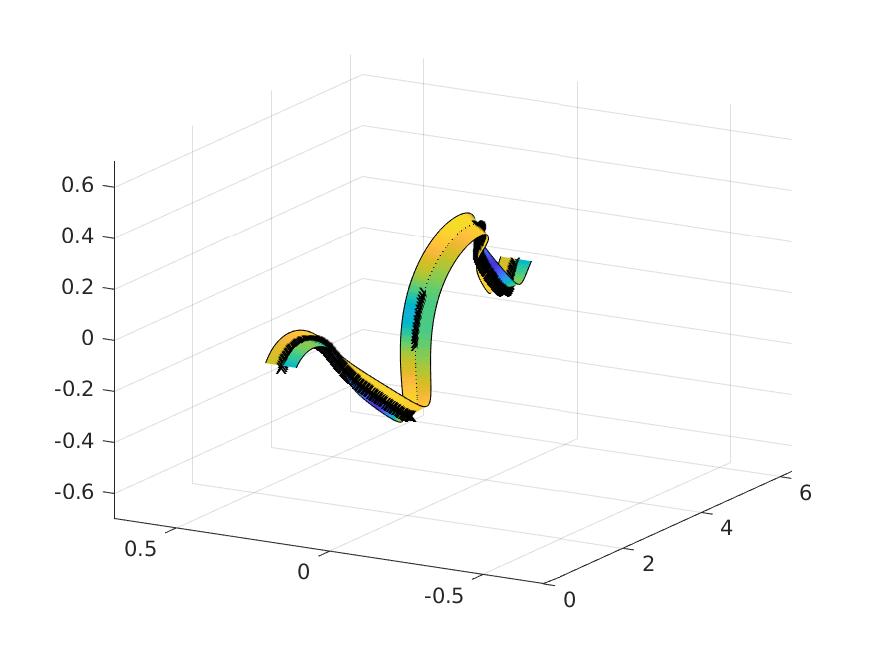} \\[-2mm]
\includegraphics[width=6.5cm]{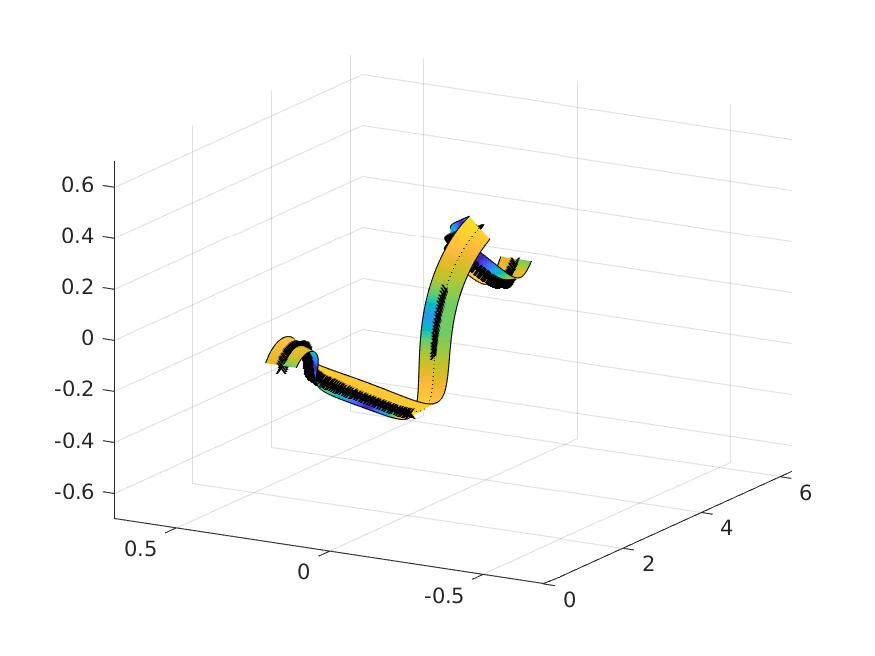} \hspace*{-8mm} 
\includegraphics[width=6.5cm]{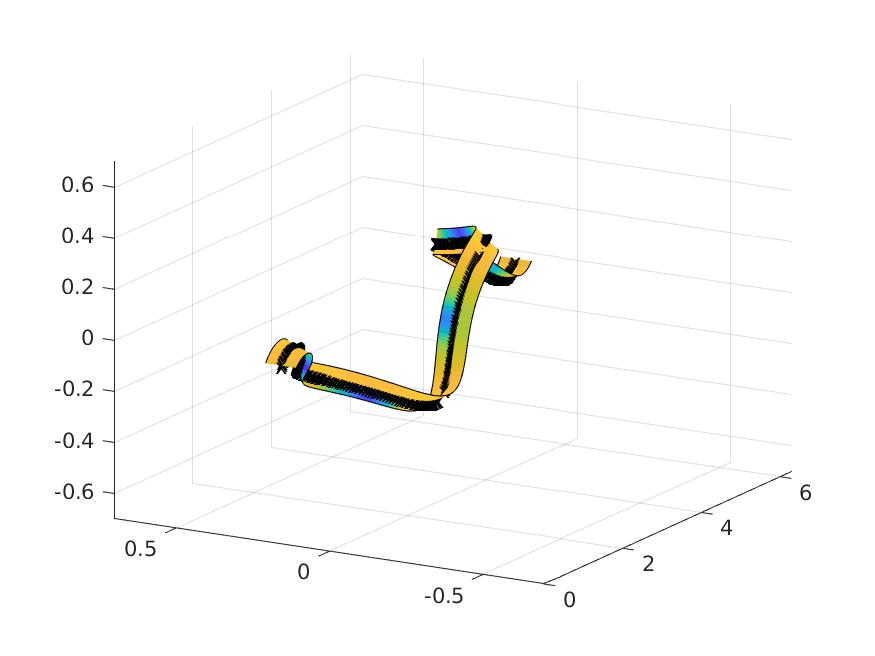} \\[-2mm]
\includegraphics[width=6.5cm]{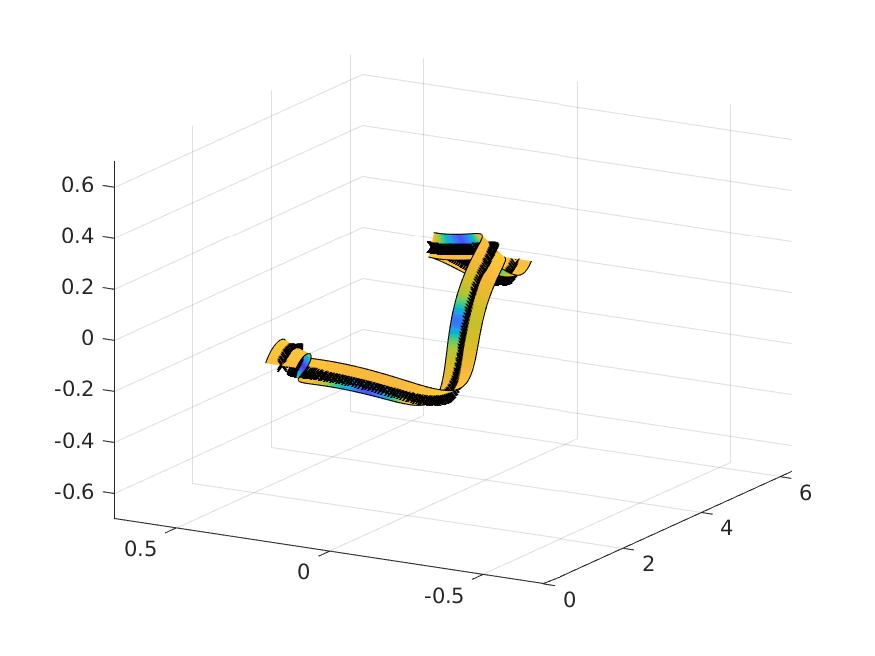} \hspace*{-8mm} 
\includegraphics[width=6.5cm]{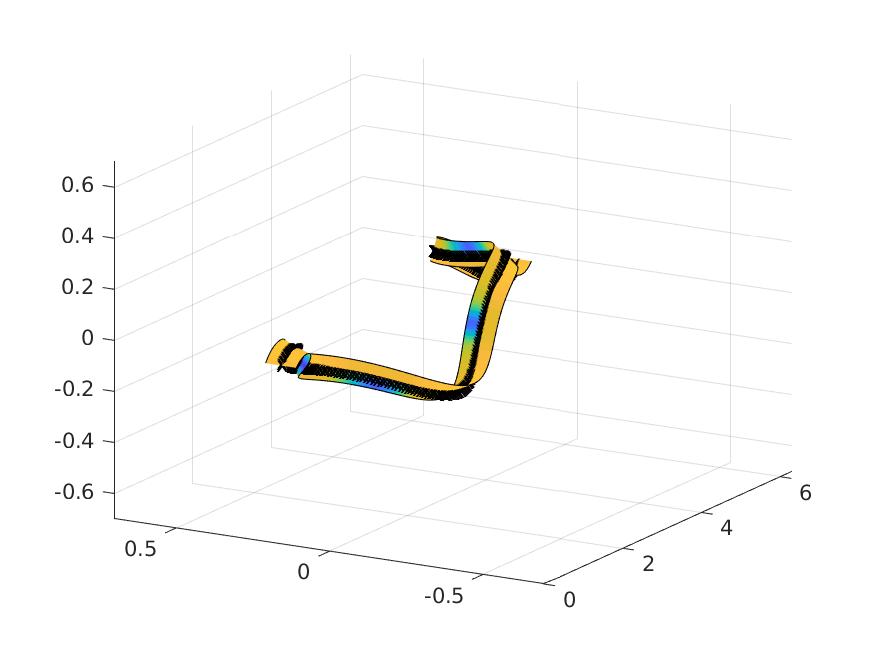}
\caption{\label{fig:band_snaps_exp_2} Snapshots of the discrete gradient flow
in Example~\ref{ex:helix} after $k=0, 408, 816, 1224, 1632, 2040, 2448, 5092$ iterations
starting from a helix, colored by curvature and torsion; 
crosses indicate where torsion dominates curvature.}
\end{figure}

\begin{figure}[p]
\includegraphics[width=5.6cm]{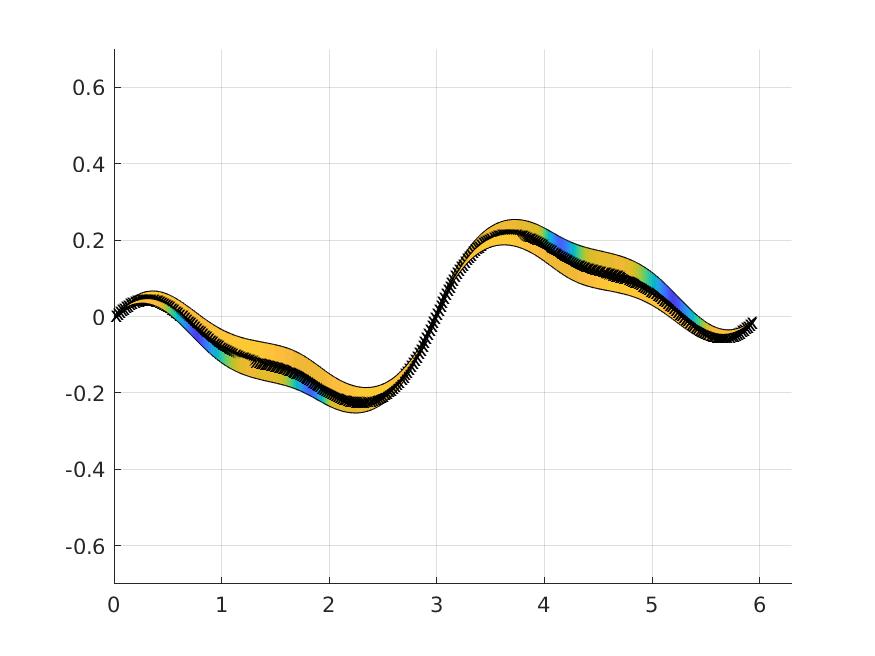}  \hspace*{-3mm}
\includegraphics[width=5.6cm]{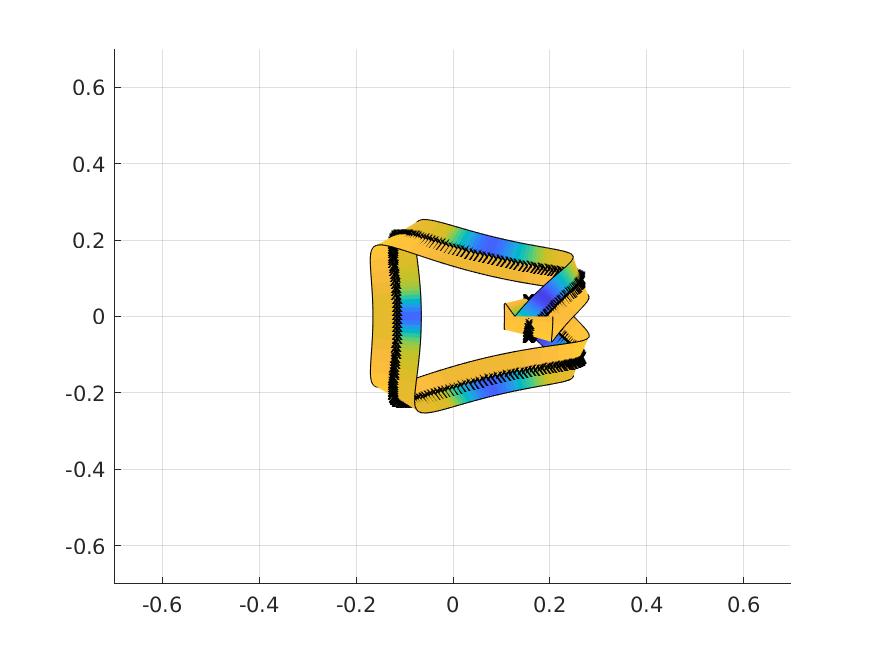} \\
\caption{\label{fig:band_finals_exp_2} Nearly stationary configuration in 
Example~\ref{ex:helix} from different perspectives colored by curvature and
torsion; crosses indicate where torsion dominates curvature.}
\end{figure}

\begin{figure}
\includegraphics[width=5.8cm]{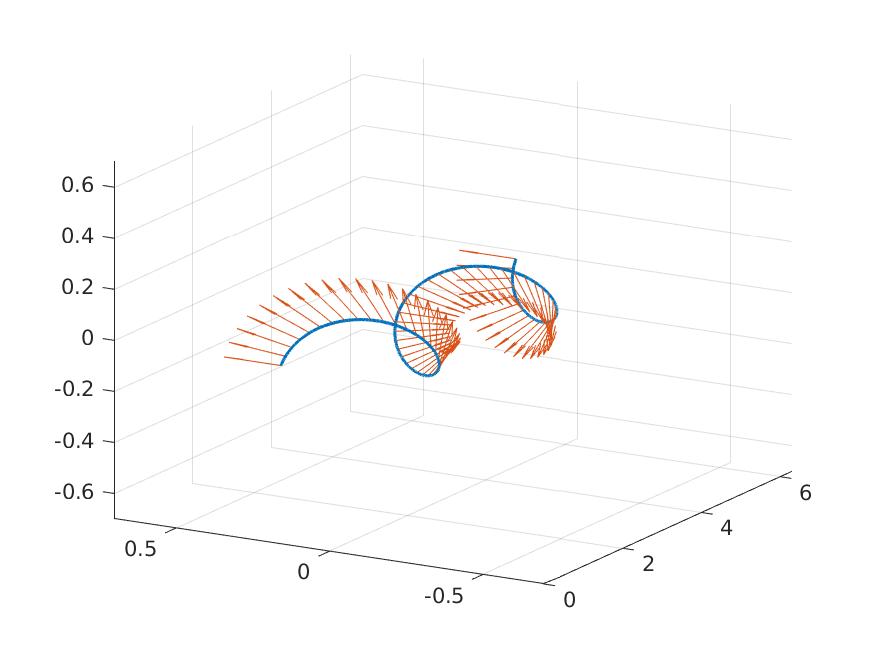} 
\includegraphics[width=5.8cm]{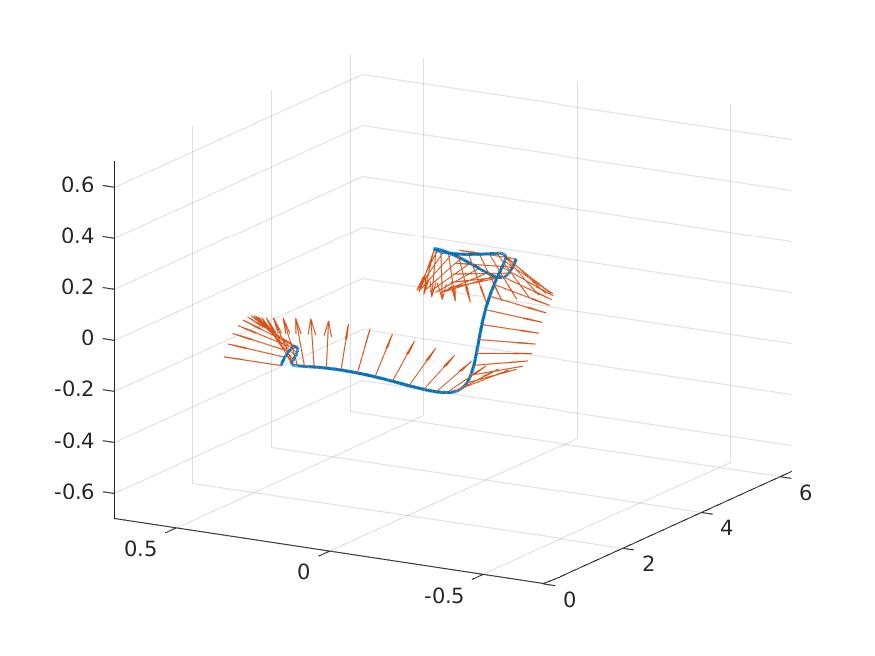} \\
\caption{\label{fig:directors_exp_2} Initial and nearly stationary director
fields on the deformed centerline 
in Example~\ref{ex:helix}; the vectors at every fourth node
are displayed.}
\end{figure}

\begin{figure}
\includegraphics[width=7.7cm]{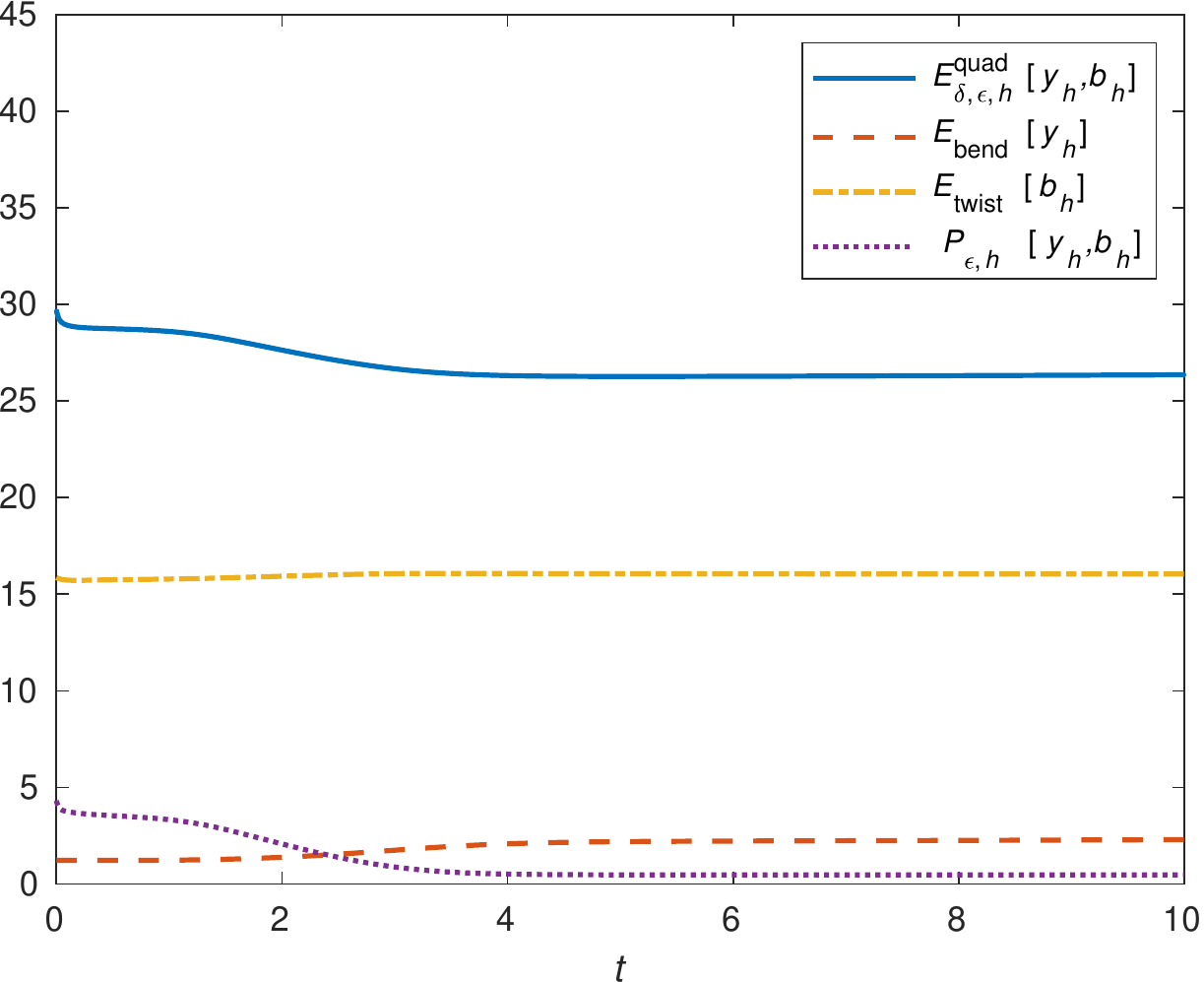} 
\caption{\label{fig:ener_exp_2} Energy decay in Example~\ref{ex:helix}
accompanied by a small increase of bending energy; twist energy and penalty functional
decay monotonically.}
\end{figure}


\medskip
\noindent
{\em Acknowledgments.} The author acknowledges stimulating discussions
with Peter Hornung. The author would like to thank the Isaac Newton 
Institute for Mathematical Sciences for support and hospitality 
during the programme {\em Geometry, compatibility and structure preservation 
in computational differential equations} when work on this paper 
was undertaken. This work was supported by EPSRC grant number EP/R014604/1.
The author also acknowledges support by the German Research Foundation (DFG) 
via the research unit FOR 3013 {\em Vector- and tensor-valued surface PDEs}.

\bibliographystyle{abbrv}
\bibliography{refs}

\end{document}